%% file: max_degree.tex
\documentclass[reqno]{amsart}

\usepackage{amsmath}
\usepackage{amsfonts}
\usepackage{amsthm}
\usepackage{amssymb}
\usepackage{graphicx}
\usepackage{graphics}
\usepackage{bm}
\usepackage{dsfont}
\usepackage{color}
\usepackage{xypic}
\usepackage[font=footnotesize]{caption}
\numberwithin{equation}{section}

\newtheoremstyle{personal}%
{12pt}
{12pt}
{\slshape}
{}
{\bfseries}
{.}
{.5em}
{}
\theoremstyle{personal}%
\newtheorem{thm}{Theorem}[section]

\newtheorem{cor}[thm]{Corollary}
\newtheorem{lem}[thm]{Lemma}
\newtheorem{prop}[thm]{Proposition}
\theoremstyle{definition}

\newtheorem{rem}{Remark}[section]

\newcommand{\N}{\mathds{N}}
\newcommand{\E}{\mathds{E}}
\newcommand{\Z}{\mathds{Z}}
\newcommand{\R}{\mathds{R}}
\newcommand{\PP}{\mathds{P}}
\newcommand{\C}{\mathds{C}}

\newcommand{\K}{\mathds{K}}

\newcommand{\Q}{\mathds{Q}}
\newcommand{\diff}{\mathrm{d}}
\newcommand{\incl}{\mathrm{incl}}
\newcommand{\dist}{\mathrm{dist}}
\newcommand{\Tan}{\mathrm{T}}
\newcommand{\qq}{\bm{q}}
\newcommand{\pp}{\bm{p}}

\newcommand{\mul}{\mathrm{mult}}
\newcommand{\crit}{\mathrm{crit}}
\newcommand{\Hom}{\mathrm{H}}
\newcommand{\Loc}{\mathrm{C}}

\newcommand{\avind}{\overline{\ind}}
\newcommand{\ind}{\mathrm{ind}}
\newcommand{\nul}{\mathrm{nul}}

\definecolor{lightgray}{gray}{0.4}
\newcommand{\gr}{\textcolor{lightgray}}

\DeclareRobustCommand{\llonghookrightarrow}{\lhook\joinrel\relbar\joinrel\relbar\joinrel\rightarrow}
\DeclareRobustCommand{\llongrightarrow}{\relbar\joinrel\relbar\joinrel\rightarrow}
\DeclareMathOperator{\rank}{\mathrm{rank}}
\DeclareMathOperator{\ev}{\mathrm{ev}} 
\DeclareMathOperator{\supp}{\mathrm{supp}} 
\DeclareMathOperator*{\toup}{\longrightarrow} 
\DeclareMathOperator*{\ttoup}{\llongrightarrow} 
 
\DeclareMathOperator*{\eembup}{\llonghookrightarrow} 
\DeclareMathOperator*{\tttoup}{\xrightarrow{\hspace*{25pt}}} 

\begin{document}

\title[Closed geodesics with local homology in maximal degree]{Closed geodesics with local homology\\ in maximal degree on non-compact manifolds}
\author{Luca Asselle}
\address{Luca Asselle\newline\indent Ruhr Universit\"at Bochum, Fakult\"at f\"ur Mathematik\newline\indent Geb\"aude NA 4/33, D-44801 Bochum, Germany}
\email{luca.asselle@ruhr-uni-bochum.de}
\author{Marco Mazzucchelli}
\address{Marco Mazzucchelli\newline\indent CNRS, \'Ecole Normale Sup\'erieure de Lyon, UMPA\newline\indent  69364 Lyon Cedex 07, France}
\email{marco.mazzucchelli@ens-lyon.fr}
\date{July 2, 2017. \emph{Revised}: November 5, 2017}

\begin{abstract}
We show that, on a complete and possibly non-compact  Riemannian manifold of dimension at least 2 without close conjugate points at infinity, the existence of a closed geodesic with local homology in maximal degree (i.e.\ in degree index plus nullity) and maximal index growth under iteration forces the existence of infinitely many closed geodesics. For closed manifolds, this was a theorem due to Hingston.
\tableofcontents
\end{abstract}
\maketitle

\vspace{-20pt}

\section{Introduction}

Since the end of the 19th century, with the work of Hadamard \cite{Hadamard:1898ww} and Poincar\'e \cite{Poincare:1905jp, Poincare:1987yj}, the study of closed geodesics in Riemannian or Finsler manifolds has occupied a central role in geometry and dynamics, and has motivated the developments of sophisticated tools in non-linear analysis (Morse theory, to mention a particularly successful one) and symplectic geometry. The intrigue in the theory comes from the fact that closed geodesics, although are often abundant in closed manifolds, are nevertheless hard to find. While the existence of one closed geodesic \cite{Birkhoff:1966xb, Lyusternik:1951jt} is a consequence of a simple (for nowadays standards) minimax trick, the existence of a second closed geodesic on a simply connected closed Riemannian manifold is a difficult problem that has been settled only in certain cases, see \cite{Bangert:2010ak, Long:2009zm} and references therein. The celebrated closed geodesics conjecture claims that every closed Riemannian manifold of dimension at least 2 possesses infinitely many closed geodesics. Currently, the conjecture is known for large classes of closed manifolds \cite{Gromoll:1969gh, Bangert:1984ni, Bangert:1993wo, Franks:1992jt}, but is still widely open for instance for higher dimensional spheres, complex projective spaces, and general compact rank-one symmetric spaces.  Evidence supporting the closed geodesics conjecture is given by a theorem due to Rademacher \cite{Rademacher:1989te, Rademacher:1994ju}, which confirms the conjecture for a $C^2$-generic Riemannian metric on any given simply connected closed Riemannian manifold of dimension at least 2.

On general complete, but non-compact, Riemannian manifolds there may be only finitely many closed geodesics (e.g.\ in cylinders with constant negative curvature), and sometimes even no closed geodesics at all (e.g.\ in flat Euclidean spaces). From the variational perspective, the ends of the manifold cause a lack of compactness for the sublevel sets of the energy function, and therefore the standard techniques from non-linear analysis fail to provide critical points. In 1992, Benci and Giannoni  \cite{Benci:1992lq} introduced a penalization method and employed it to show the existence of a closed geodesic on every complete Riemannian manifold with a certain minimum of loop space homology, and under a suitable condition on the geometry at infinity: the Riemannian metric must have non-positive curvature at infinity. In the recent paper \cite{Asselle:2017pd}, the authors improved Benci and Giannoni's result as follows: every $d$-dimensional complete Riemannian manifold \emph{without close conjugate points at infinity} has at least a closed geodesic provided its free loop space homology is nontrivial in some degree larger than $d$, and infinitely many closed geodesics provided the set of Betti numbers of its free loop space homology in degrees larger than $d$ is unbounded. The condition of not having close conjugate points at infinity means that, for every length $\ell>0$, there exists a compact subset of the manifold outside which no pairs of points are conjugate along geodesics of length less than or equal to $\ell$.

On closed Riemannian manifolds of dimension at least 2, Hingston \cite{Hingston:1993ou} showed that the presence of a particular kind of closed geodesic -- one with local homology in maximal degree and maximal index growth under iteration -- forces the existence of infinitely many more closed geodesics. This remarkable and hard theorem is one of the crucial ingredients for proving that the growth rate of closed geodesics on any Riemannian 2-sphere is at least the one of prime numbers. 
Our main result extends the validity of Hingston's theorem to non-compact complete Riemannian manifolds without close conjugate points at infinity. 

\begin{thm}\label{t:main}
Let $(M,g)$ be a complete Riemannian manifold of dimension at least 2 without close conjugate points at infinity and possessing a non-iterated closed geodesic $\gamma$ such that:
\begin{itemize}
\item[(i)] the local homology $\Loc_d(E,\gamma)$ is non-trivial in degree $d=\avind(\gamma)+\dim(M)-1$;
\item[(ii)] $\ind(\gamma^m)+\nul(\gamma^m)=m\,\avind(\gamma)+\dim(M)-1$ for all $m\in\PP\cup\{1\}$, where $\PP\subset\N$ is some infinite subset of prime numbers.
\end{itemize}
Then $(M,g)$ contains infinitely many closed geodesics.
\end{thm}

We refer the reader to Sections~\ref{s:Benci_Giannoni} and~\ref{s:local_homology} for the background on the Morse index and the local homology of closed geodesics.
Hingston showed that the above conditions~(i-ii) imply that each iterate $\gamma^m$, for $m\in\PP\cup\{1\}$, has non-trivial local homology in maximal degree $\ind(\gamma^m)+\nul(\gamma^m)$, and trivial local homology in all other degrees. This will be explained with full details also in Section~\ref{ss:construction}. One can also show that that the closed geodesic $\gamma$ must be degenerate, and indeed must have only 1 as Floquet multiplier.

An earlier result of Bangert and Klingenberg \cite[Theorem~3]{Bangert:1983ax}, along the line of Hingston's one, stated that a closed Riemannian manifold $(M,g)$ has infinitely many closed geodesics every time it has one, say $\gamma$, with average index $\avind(\gamma)=0$, non-trivial local homology, and that is not a global minimizer of the energy function in its connected component. This theorem has an intersection with Hingston's one: the two results cover the case in which the special closed geodesic $\gamma$ has average index $\avind(\gamma)=0$ and non-trivial local homology $\Loc_d(E,\gamma)$ in degree $d=\dim(M)-1\geq1$. Remarkably, our argument for proving Theorem~\ref{t:main} does not allow to extend the remaining cases in Bangert and Klingenberg's full theorem (that is, when $d\in\{0,...,\dim(M)-2\}$) to general non-compact manifolds without close conjugate points at infinity.

In the context of Hamiltonian diffeomorphisms of certain closed symplectic manifolds, the analog of the closed geodesic $\gamma$ of Theorem~\ref{t:main} in the special case where $\avind(\gamma)=0$ are the so called symplectically degenerate maxima: these are periodic points with zero average Maslov index and non-trivial local homology in maximal degree. A result first proved by Hingston \cite{Hingston:2009hp} for standard symplectic tori, and further extended by Ginzburg to closed aspherical symplectic manifolds \cite{Ginzburg:2010en}, implies that the existence of a symplectically degenerate maximum forces the existence of infinitely many other periodic points whose growth rate is at least the one of prime numbers. This theorem is a central ingredient in the proof of the Conley conjecture in symplectic geometry: the Hamiltonian diffeomorphisms of certain closed symplectic manifolds always have infinitely many periodic points, see \cite{Hingston:2009hp, Ginzburg:2010en, Hein:2012ml, Ginzburg:2012nx, Ginzburg:2010wh, Mazzucchelli:2013co} and references therein.

Certain theorems and conjectures in the theory of closed geodesics became infamous due to the remarkable number of mistakes in the published literature. Just to mention a couple of examples, the 1929 Lusternik and Schnirelmann's theorem \cite{Lusternik:1934km} on the existence of three simple closed geodesics on any Riemannian 2-sphere was universally accepted as correct only after the work of several authors in the late 1970s and 1980s (see e.g.\ \cite{Ballmann:1978rw, Grayson:1989gd}), and the above mentioned closed geodesics conjecture for closed Riemannian manifolds appeared as a theorem with an erroneous proof in Klingenberg's monograph \cite{Klingenberg:1978so}. In view of these accidents, in the current paper we decided to devote extra care and add full details in the proofs of several statements that would otherwise be considered known or folklore by the community. Along the way, we have taken the occasion to correct some minor technical mistakes that are recurrent in the closed geodesics literature.

\subsection{Organization of the paper} 
In Section~\ref{s:setting} we will recall the variational setting for the problem of closed geodesics on complete manifolds without close conjugate points at infinity. Section~\ref{s:local_homology} will be devoted to present and clarify known results about the local homology of closed geodesics; certain general properties of the local homology of abstract functions that we will make use of and that we could not find in the literature are collected in the appendix of the paper. In Section~\ref{s:idx_based} we will prove an iteration inequality for the Morse index of closed geodesics in the based loop space, which will be needed later on in the paper. In Section~\ref{s:avind_positive} we will prove Theorem~\ref{t:main} in the case where the given close geodesic has positive average index, whereas the case of vanishing average index will be treated in Section~\ref{s:avind_zero}.

\subsection{Acknowledgements}
The authors are grateful to an anonymous referee for her/his helpful remarks on the first draft of this paper.  Luca Asselle is partially supported by the DFG-grant AB 360/2-1 ``Periodic orbits of conservative systems below the Ma\~n\'e critical energy value''. Marco Mazzucchelli is partially supported by the ANR COSPIN (ANR-13-JS01-0008-01).

\section{The variational setting of closed geodesics}
\label{s:setting}

\subsection{The energy function}

Let $(M,g)$ be a Riemannian manifold. Its closed geodesics are the critical points with positive critical value of the smooth energy function
\begin{align*}
E:\Lambda M\to[0,\infty),\qquad
E(\gamma):=\int_0^1 g_{\gamma(t)}(\dot\gamma(t),\dot\gamma(t))\,\diff t,
\end{align*}
where $\Lambda M:=W^{1,2}(\R/\Z,M)$ is the Sobolev free loop space of $M$. We recall that the unit circle $S^1=\{e^{2\pi is}\in\C\ |\ s\in\R\}$ acts on the free loop space $\Lambda M$ by reparametrization:
\begin{align*}
e^{2\pi is}\cdot\gamma = \gamma(s+\cdot),\qquad
e^{2\pi is}\in S^1,\ \gamma\in\Lambda M.
\end{align*}
For each $\theta\in S^1$, the induced map $\theta:\Lambda M\to\Lambda M$, $\gamma\mapsto\theta\cdot\gamma$ is a diffeomorphism that leaves the energy function invariant, that is, $E(\theta\cdot\gamma)=E(\gamma)$. In particular, every closed geodesic $\gamma$ belongs to a critical circle $S^1\cdot\gamma$ of $E$.

If we equip the Hilbert manifold $\Lambda M$ with the usual Riemannian metric $G$ induced by $g$, the Hessian of the energy $E$ at a closed geodesic $\gamma$ is defined by an operator $H:\Tan_\gamma\Lambda M\to\Tan_\gamma\Lambda M$, i.e.
\begin{align*}
\diff^2E(\gamma)[\xi,\eta]
=
G_\gamma(H\xi,\eta),
\qquad
\forall \xi,\eta\in\Tan_\gamma\Lambda M.
\end{align*}
A computation shows that $H$ is a compact perturbation of the identity, see \cite[Lemma~3]{Gromoll:1969gh}. In particular, its negative eigenspace $\E_-$ and its kernel $\E_0$ are finite dimensional. Since $\gamma$ belongs to the critical circle $S^1\cdot\gamma$, the vector field $\dot\gamma\in\Tan_{\gamma}\Lambda M$ belongs to $\E_0$, and therefore $\dim(E_0)\geq1$. The Morse index and the nullity of $E$ at $S^1\cdot\gamma$ are defined by
\begin{align*}
\ind(\gamma) & =\ind(E,S^1\cdot\gamma)=\dim(\E_-),\\
\nul(\gamma) & =\nul(E,S^1\cdot\gamma)=\dim(\E_0)-1,
\end{align*}
or, equivalently,
\begin{align*}
\ind(\gamma) & := \max\big\{ \dim\E\ \big|\ \E\subset\Tan_\gamma\Lambda M\mbox{ with }\diff^2 E(\gamma)[\xi,\xi]<0\ \forall\xi\in\E\setminus\{0\}\big\},\\
\ind(\gamma)+\nul(\gamma) & := \max\big\{ \dim\E\ \big|\ \E\subset\Tan_\gamma\Lambda M\mbox{ with }\diff^2 E(\gamma)[\xi,\xi]\leq 0\ \forall\xi\in\E\big\}-1.
\end{align*}

If $\nul(\gamma)=0$, by Gromoll-Meyer's generalized Morse Lemma  \cite[Page~501]{Gromoll:1969gh}, these two indices completely determine the germ of the energy function $E$ at $S^1\cdot\gamma$ up to a diffeomorphism of the domain. When $\nul(\gamma)>0$, more insight on the critical circle $S^1\cdot\gamma$ is provided by its local homology
\begin{align*}
\Loc_*(E,\gamma)
:=\Hom_*(\{E<E(\gamma)\}\cup S^1\cdot\gamma\},\{E<E(\gamma)\}),
\end{align*}
where $\{E<E(\gamma)\}$ is an abbreviation for the sublevel set $\{\zeta\in\Lambda M\ |\ E(\zeta)<E(\gamma)\}$. Throughout this paper, $\Hom_*$ denotes the singular homology with rational coefficients.

Every closed geodesic $\gamma$ contributes countably many critical circles of the energy function (a ``tower'' of critical circles, in the classical terminology). Indeed, for every integer $m\in\N=\{1,2,3,...\}$, the $m$-th iterare $\gamma^m\in\Lambda M$ of the original $\gamma$, which is given by $\gamma^m(t)=\gamma(mt)$, belongs to a critical circle $S^1\cdot\gamma^m$ of $E$. In order to recognize when two critical circles belong to the same tower, it is useful to know the behavior of the functions $m\mapsto\ind(\gamma^m)$, $m\mapsto\nul(\gamma^m)$. A recipe for computing these functions was found by Bott \cite{Bott:1956sp} in the 1950s. Here, we simply recall that the first function grows linearly, and more precisely
\begin{gather}
\label{e:iteration_inequality_1}
m\,\avind(\gamma)-(\dim(M)-1)
\leq
\ind(\gamma^m),\\
\label{e:iteration_inequality_2}
\ind(\gamma^m)+\nul(\gamma^m)
\leq 
m\,\avind(\gamma)+\dim(M)-1,
\end{gather}
where $\avind(\gamma)\in(0,\infty)$ is the average index of $\gamma$, defined by
\begin{align*}
\avind(\gamma):=\lim_{m\to\infty} \frac{\ind(\gamma^m)}{m}.
\end{align*}
This shows in particular that condition~(ii) in Theorem~\ref{t:main} requires precisely that the function $m\mapsto\ind(\gamma^m)+\nul(\gamma^m)$ has maximal growth.

We refer the reader to \cite{Klingenberg:1978so, Klingenberg:1995fv, Long:2002ed} and references therein for more background on the variational setting of the closed geodesics problem.

\subsection{Benci-Giannoni's penalization method}
\label{s:Benci_Giannoni}

When $M$ is a closed manifold, $E$ is well-behaved and allows one to apply the techniques from abstract critical point theory, see e.g. \cite{Klingenberg:1978so} and references therein. On the contrary, when $M$ is complete but not closed, the sublevel sets of $E$ may not have the minimal compactness required by critical point theory. Following Benci and Giannoni \cite{Benci:1992lq}, a way to regain such a compactness is to add a suitable penalization to $E$ as follows. We choose any family of proper smooth functions $\{f_\alpha:M\to[0,\infty)\ |\ \alpha\in\N\}$ such that $f_0\geq f_1\geq f_2\geq ...$ pointwise, and for all compact sets $K\subset M$ there exists $\alpha_0=\alpha_0(K)\in\N$ such that 
$\supp(f_{\alpha_0})\cap K=\varnothing$. One can easily build such a family of functions by means of a partition of unity. The penalized energy functions
\begin{align*}
E_\alpha:\Lambda M\to\R,\qquad
E_\alpha(\gamma):=E(\gamma)+f_\alpha(\gamma(0))
\end{align*}
are smooth and satisfy the Palais-Smale condition, see~\cite[Lemma~4.5]{Benci:1992lq}. A loop $\gamma\in\Lambda M$ is a critical point of $E_\alpha$ if and only if $\gamma|_{(0,1)}$ is a geodesic and
\begin{align*}
 \dot\gamma(0^-)- \dot\gamma(0^+) = \mathrm{grad}f_\alpha(\gamma(0)),
\end{align*}
where the gradient in the right-hand side is the one induced by the Riemannian metric $g$. In other words, the critical points of $E_\alpha$ are the closed geodesics $\gamma$ whose base point $\gamma(0)$ lies outside the support of $f_\alpha$, and certain geodesic loops based at points inside the support of $f_\alpha$. Benci and Giannoni's clever idea was that these latter critical points can be suitably ``filtered out'' if one requires the complete Riemannian manifold $(M,g)$ to have non-positive curvature at infinity. Indeed, for each $b>0$ there exists $\alpha=\alpha(b)\in\N$ large enough such that each critical point $\gamma$ of $E_\alpha$ with $E_\alpha(\gamma)\leq b$ and $\gamma(0)\in\supp(f_\alpha)$ is entirely contained in the region of $M$ where the curvature is non-positive, and in particular its Morse index and nullity satisfy $\ind(E_\alpha,\gamma)+\nul(E_\alpha,\gamma)\leq\dim(M)$. This implies that all minmax procedures performed with relative cycles of degree $d>\dim(M)$ detect critical points $\gamma$ of $E_\alpha$, for $\alpha$ large enough, that are genuine closed geodesics; indeed, such $\gamma$'s satisfy $\ind(E_\alpha,\gamma)+\nul(E_\alpha,\gamma)\geq d>\dim(M)$, and  thus $\gamma(0)$ must lie outside the support of $f_\alpha$.

In \cite{Asselle:2017pd}, the authors realized that Benci and Giannoni's idea works under a weaker assumption on the complete Riemannian metric: it is enough to require $(M,g)$ to be without close conjugate points at infinity. The precise argument goes as follows. Since in the applications we will be looking for infinitely many closed geodesics of $(M,g)$, we can assume that they are all isolated (i.e.\ they belong to isolated critical circles of the energy function $E$) and contained in a relatively compact open subset $W\subset M$. For all $\alpha\in\N$ large enough, say $\alpha\geq\alpha_0(\overline{W})$, the support of $f_\alpha$ does not intersect the closure of $W$. We set 
\[\mathcal{U}_\alpha:=\big\{\gamma\in\Lambda M\ \big|\ \gamma(0)\in\supp(f_\alpha)\big\},\]
so that $\Lambda W\cap\mathcal{U}_\alpha=\varnothing$. Since the Riemannian manifold $(M,g)$ is without close conjugate points at infinity, for each $\ell>0$ there exists $\alpha\in\N$ large enough, say $\alpha\geq\alpha_1(\ell)\geq\alpha_0(\overline{W})$, such that no pair of points $q_0,q_1\in \supp(f_\alpha)$ is conjugated along a geodesic of length at most $2\ell$. In particular, all the critical points $\gamma\in\crit(E_\alpha)\cap\mathcal{U}_\alpha$ with $E(\gamma)\leq (2\ell)^2$ are geodesic loops such that no point $\gamma(t)$ with $t\in(0,1]$ is conjugate to $\gamma(0)$ along $\gamma|_{[0,t]}$; therefore, by~\cite[Lemma~2.1]{Asselle:2017pd}, the Morse indices of $E_\alpha$ at $\gamma$ satisfy
\begin{align*}
\ind(E_\alpha,\gamma)+\nul(E_\alpha,\gamma)\leq\dim(M).
\end{align*}
By a standard argument in Morse theory \cite{Marino:1975ii}, we can perturb in a $C^2$-small fashion the function $E_\alpha$ on an arbitrarily small neighborhood of $\crit(E_\alpha)\cap\mathcal{U}_\alpha$ and resolve all the degeneracies of the critical points in $\mathcal{U}_\alpha$. Thus, we obtain a smooth function $F_\alpha:\Lambda M\to\R$ that is arbitrarily $C^2$-close to $E_\alpha$, coincides with $E_\alpha$ outside $\mathcal{U}_\alpha$ (in particular, inside $\Lambda W$), and each critical point $\gamma\in\crit(F_\alpha)\setminus\Lambda W$ with critical value $E(\gamma)
\leq(2\ell)^2$ is non-degenerate and has Morse index $\ind(F_\alpha,\gamma)\leq\dim(M)$.

For all $0< a<b$, we set $\mathcal G(a,b)$ to be the collection of the critical circles of closed geodesics $\gamma$ of $(M,g)$ with $E(\gamma)\in(a,b)$, that is,
\begin{align*}
\mathcal G(a,b)
:=
\big\{
S^1\cdot\gamma
\ \big|\ 
\gamma\in\crit(E)\cap E^{-1}(a,b)
\big\}.
\end{align*}
For each $\alpha\geq\alpha_1(b^{1/2})$, we also define 
\begin{align*}
\mathcal G_\alpha(a,b)
:=\crit(F_\alpha)\cap F_\alpha^{-1}(a,b)\setminus \Lambda W,
\end{align*}
which is a (discrete) set of non-degenerate critical points of $F_\alpha$.
Notice that
\begin{align*}
\crit(F_\alpha)\cap F_\alpha^{-1}(a,b) = \mathcal G(a,b) \cup \mathcal G_\alpha(a,b).
\end{align*}
The set of non-degenerate critical points $\mathcal G_\alpha(a,b)$ is discrete. Each $\gamma\in\mathcal G_\alpha(a,b)$ carries local homology
\begin{align*}
\Loc_d(F_\alpha,\gamma)
=
\Hom_d(\{F_\alpha<F_\alpha(\gamma)\}\cup\{\gamma\},\{F_\alpha<F_\alpha(\gamma)\})
=
\left\{
  \begin{array}{@{}ll}
    \Q &   \mbox{if }\ind(E_\alpha,\gamma)=d, \\ 
    0 &   \mbox{otherwise}, 
  \end{array}
\right.
\end{align*}
and we have the Morse inequality
\begin{align*}
\rank \Hom_d(\{F_\alpha<b\},\{F_\alpha\leq a\})
\leq &
\sum_{S^1\cdot\gamma\subset\mathcal G(a,b)} \!\!\!\!\!\!\!\rank\Loc_d(E,S^1\cdot\gamma)\\
& +
\sum_{\gamma\in\mathcal G_\alpha(a,b)} \!\!\!\!\!\rank\Loc_d(F_\alpha,\gamma)\\
= &
\sum_{S^1\cdot\gamma\subset\mathcal G(a,b)} \!\!\!\!\!\!\!\rank\Loc_d(E,S^1\cdot\gamma)\\
& +
\#
\big\{
\gamma\in\mathcal G_\alpha(a,b)\ \big|\ \ind(F_\alpha,\gamma)= d
\big\}
\end{align*}
If we consider degrees $d>\dim(M)$, the latter summand vanishes, and therefore
\begin{equation*}
\begin{split}
\rank \Hom_d(\{F_\alpha<b\},\{F_\alpha\leq a\})
\leq 
\sum_{S^1\cdot\gamma\subset\mathcal G(a,b)} \!\!\!\!\!\!\!\rank\Loc_d(E,S^1\cdot\gamma),
\\
\forall d>\dim(M).
\end{split}
\end{equation*}

\section{The local homology of closed geodesics}
\label{s:local_homology}

Throughout this section, we follow Bangert and Long's~\cite[Section~3]{Bangert:2010ak} closely. We will provide full proofs of all the statements, since on the one hand some of these proofs are different from those given in~\cite[Section~3]{Bangert:2010ak}, and on the other hand  for later purposes we will need to be more explicit on the maps that realize certain homology isomorphisms. 

As before, we will  denote by $\Hom_*$ the singular homology with rational coefficients. We will also adopt the following notation: given two topological pairs $(X,A)\subset(Y,B)$, which means that $A\subseteq B\cap X$ and $B\subseteq Y$, we will denote the homology homomorphism induced by the inclusion as
\begin{align*}
\Hom_*(X,A)\ttoup^{\incl} \Hom_*(Y,B).
\end{align*}

\subsection{The space of closed broken geodesics}

Let $(M,g)$ be a Riemannian manifold. We denote by $\rho_{\mathrm{inj}}(q)$ the injectivity radius of $(M,g)$ at a point $q\in M$. We say that a geodesic $\zeta:[a,b]\to M$ is short when its length is strictly smaller than $\rho_{\mathrm{inj}}(\zeta(a))$. For each integer $k\geq2$, we introduce the space of closed broken geodesics
\begin{align*}
\Lambda_kM:=\big\{\gamma\in\Lambda M \ \big|\ \gamma|_{[i/k,(i+1)/k]} \mbox{ is a short geodesic } \forall i=0,...,k-1\big\},
\end{align*}
which is diffeomorphic to an open subset of the $k$-fold product $M\times...\times M$ via the evaluation map $\gamma\mapsto(\gamma(0),\gamma(\tfrac 1k),...,\gamma(\tfrac{k-1}{k}))$.

Let $E:\Lambda M\to\R$ be the energy function of the Riemannian manifold $(M,g)$, and $S^1\cdot\gamma\subset\crit(E)\cap E^{-1}(0,\infty)$ the critical circle of a closed geodesic $\gamma$. Let $k\in\N$ be large enough so that $S^1\cdot\gamma\subset\Lambda_kM$. For a sufficiently small $S^1$-invariant neighborhood $\mathcal U\subset\Lambda M$ of $S^1\cdot\gamma$, we can define a homotopy
\begin{align}\label{e:def_retraction}
r_s:\mathcal U\to\Lambda M,\qquad s\in[0,1],
\end{align}
as follows: for each $\zeta\in\mathcal U$, the closed curve $r_s(\zeta)$ coincides with $\zeta$ everywhere except on intervals of the form $[\tfrac{i}{k},\tfrac{i+s}{k}]$, for $i=0,...,k-1$, and the restrictions of  $r_s(\zeta)$ to such intervals are short geodesics. Notice that $r_0$ is the identity, and each map $r_s$, for $s\in(0,1]$, fixes $\mathcal U\cap\Lambda_kM$ and does not increase the energy, that is, $E\circ r_s\leq E$. Moreover, the image $U:=r_1(\mathcal U)$ is a  neighborhood of $S^1\cdot\gamma$ in $\Lambda_kM$. The time-1 map 
\begin{align}\label{e:retraction}
r:\mathcal U\to U
\end{align}
sends any $\zeta\in\mathcal{U}$ to the unique closed broken geodesic $\zeta':=r(\zeta)$ such that $\zeta'(\tfrac ik)=\zeta(\tfrac ik)$ for all $i=0,...,k-1$.

The following proposition was essentially stated without proof by Gromoll and Meyer \cite[Page~501]{Gromoll:1969gh} in the infinite dimensional setting of $\Lambda M$. The statement given here was provided by Bangert and Long\footnote{The proof of \cite[Lemma~3.5]{Bangert:2010ak} contains a minor inaccuracy: the curve $\Gamma_\gamma(s)$ introduced there is not differentiable at $s=1$, and therefore it cannot be employed in \cite[Equation~(3.14)]{Bangert:2010ak}.} \cite[Lemma~3.5]{Bangert:2010ak}.

\begin{prop}\label{p:isolated}
Let $\gamma\in\crit(E)\cap\Lambda_kM$ be a closed geodesic, and $\Sigma\subset\Lambda_kM$ a hypersurface intersecting the critical circle $S^1\cdot\gamma$ transversely at $\gamma$. If  $S^1\cdot\gamma$ is isolated in $\crit(E)$, then $\gamma$ is isolated in $\crit(E|_{\Sigma})$.
\end{prop}

\begin{proof}
Assume that $S^1\cdot\gamma$ is isolated in $\crit(E)$, and consider the open neighborhoods $\mathcal U\subset\Lambda M$ and $U=r_1(\mathcal U)\subset\Lambda_kM$  of $S^1\cdot\gamma$ introduced above. Let $\mathcal{V}\subset\mathcal{U}$ be a small enough neighborhood of $S^1\cdot\gamma$ so that 
\begin{align*}
\mathcal V\cap\crit(E)=S^1\cdot\gamma.
\end{align*}
The intersection $V:=U\cap\mathcal{V}$ is an open neighborhood of $S^1\cdot\gamma$ in $\Lambda_kM$. After shrinking $\Sigma$ around $\gamma$, we can assume that $\Sigma\subset V$ and $\Sigma\cap S^1\cdot\gamma=\{\gamma\}$. For each $\zeta\in\Sigma$, we define the curve 
\begin{align*}
Z:(-\tfrac1k,\tfrac1k)\to\Lambda_k M,
\qquad
Z(s)=r(e^{2\pi is}\cdot\zeta).
\end{align*}
Notice that $Z$ is a continuous curve in $\Lambda_kM$, and is smooth outside $0$. Indeed, for each $i\in\{0,...,k-1\}$, the curve $s\mapsto r(e^{2\pi is}\cdot\zeta)(\tfrac ik)$ is smooth outside $s=0$. Moreover, $Z$ is smooth at $0$ if and only if $\zeta$ is a smooth curve, and thus a closed geodesic. Hence $Z$ is smooth at $0$ if and only if $\zeta=\gamma$.

\begin{figure}
\begin{center}
\begin{footnotesize}
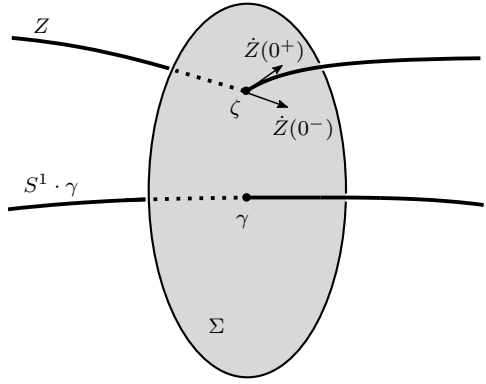 
\end{footnotesize}
\begin{small}
\caption{The hypersurface $\Sigma$ transverse to the critical circle $S^1\cdot\gamma$.}
\label{f:section}
\end{small}
\end{center}
\end{figure}

Assume now that $\zeta\neq\gamma$. If $\zeta$ is close to $\gamma$ in $\Sigma$, each restriction $\zeta|_{[i/k,(i+1)/k]}$ is $C^\infty$-close to $\gamma|_{[i/k,(i+1)/k]}$. This implies that the tangent vectors $\dot Z(0^+)$ and $\dot Z(0^-)$ are both transverse to $\Tan_\zeta\Sigma$ and point in the same direction, that is,
\begin{align*}
 \dot Z(0^-)=\lambda\dot Z(0^+)+\nu
\end{align*}
for some $\lambda>0$ and $\nu\in\Tan_\zeta\Sigma$, see Figure~\ref{f:section}. Moreover, the function $s\mapsto E(Z(s))$ has a strict maximum at $s=0$, and is smooth separately on both intervals $(-\tfrac1k,0]$ and $[0,\tfrac1k)$. Therefore
\begin{equation}\label{e:transverse_variations}
\begin{split}
\diff E(\zeta)\dot Z(0^+) = \tfrac{\diff}{\diff s}\big|_{s=0^+} E(Z(s))\leq 0,\\
\diff E(\zeta)\dot Z(0^-) = \tfrac{\diff}{\diff s}\big|_{s=0^-} E(Z(s))\geq 0. 
\end{split}
\end{equation}

Assume that at least one of these two inequalities is not strict, that is, there exists $\xi\in\{\dot Z(0^+),\dot Z(0^-)\}$ such that $\diff E(\zeta)\xi=0$. Since $\diff E(\zeta)|_{\Tan_\zeta\Lambda_kM}$ is non-zero, the kernel $\ker(\diff E(\zeta)|_{\Tan_\zeta\Lambda_kM})$ is a hyperplane in $\Tan_\zeta\Lambda_kM$ of codimension $1$. Since the tangent vector $\xi$ belongs to $\ker(\diff E(\zeta)|_{\Tan_\zeta\Lambda_kM})$ and is transverse to the hyperplane $\Tan_\zeta\Sigma$, we conclude that $\Tan_\zeta\Sigma$ is not contained in $\ker(\diff E(\zeta)|_{\Tan_\zeta\Lambda_kM})$, and thus $\zeta$ is not a critical point of $E|_\Sigma$.

Assume now that both inequalities in~\eqref{e:transverse_variations} are strict. This implies that
\begin{align*}
\diff E(\zeta)\nu
=
\diff E(\zeta)\dot Z(0^-) - \lambda\,\diff E(\zeta)\dot Z(0^+)>0.
\end{align*}
Therefore, the tangent vector $\nu\in\Tan_\zeta\Sigma$ does not belong to the kernel of $\diff E(\zeta)$, which again implies that $\zeta$ is not a critical point of $E|_\Sigma$.
\end{proof}

\subsection{Local homology}

Any smooth connected submanifold $K\subset \crit(E)$ is contained in some level set $E^{-1}(c)$. The local homology of $K$ is defined as the relative homology group
\begin{align*}
\Loc_*(E,K):=\Hom_*(\{E<c\}\cup K,\{E<c\}).
\end{align*}

\begin{lem}\label{l:local_homology_and_shift}
Let $\gamma\in\crit(E)\cap\Lambda_kM$ be a closed geodesic. 
\begin{itemize}
\item[(i)] For each hypersurface $\Sigma\subset \Lambda_kM$ intersecting the critical circle $S^1\cdot\gamma$ transversely at $\gamma$, the inclusion induces a local homology isomorphism
\begin{align*}
\Loc_*(E|_{\Sigma},\gamma)
\ttoup_{\cong}^{\incl}
\Loc_*(E,\gamma).
\end{align*}

\item[(ii)] For each embedded open interval $K\subset S^1\cdot\gamma$ containing $\gamma$, the inclusion induces a local homology isomorphism
\begin{align*}
\Loc_*(E,\gamma)
\ttoup_{\cong}^{\incl}
\Loc_*(E,K).
\end{align*}

\item[(iii)] For each $\theta=e^{2\pi i t}\in S^1$, the corresponding map on  $\Lambda M$ induces local homology isomorphisms, and for each embedded open interval $K\subset S^1\cdot\gamma$ containing  $e^{2\pi i t s}\cdot\gamma$ for all $s\in[0,1]$  we have a commutative diagram
\begin{align*}
\xymatrix{
\Loc_*(E,\gamma)
\ar[rr]_{\cong}^{\theta_*}
\ar[ddrr]_{\incl}^\cong
&&
\Loc_*(E,\theta\cdot\gamma)
\ar[dd]^{\incl}_\cong
\\\\
&&
\Loc_*(E,K)
} 
\end{align*}
\end{itemize}
\end{lem}

\begin{proof}
It is well known that the inclusion induces a local homology isomorphism
\begin{align*}
\Loc_*(E|_{\Lambda_kM},\gamma)
\ttoup_{\cong}^{\incl}
\Loc_*(E,\gamma).
\end{align*}
This is a consequence of Proposition~\ref{p:homotopic_invariance_local_homology} applied to the triple of spaces 
\[\{\gamma\}\subset \{E|_{\Lambda_kM}<c\}\cup\{\gamma\} \subset \{E<c\}\cup\{\gamma\},\]
where $c:=E(\gamma)$, and to the homotopy $r_s:\mathcal U\to\Lambda_kM$ introduced in~\eqref{e:def_retraction}. Therefore, in order to prove point~(i) it is enough to show that the inclusion induces a local homology isomorphism
\begin{align*}
\Loc_*(E|_{\Sigma},\gamma)
\ttoup_{\cong}^{\incl}
\Loc_*(E|_{\Lambda_kM},\gamma).
\end{align*}

Let $r:\mathcal U\to U$ be the time-1 map of  Equation~\eqref{e:retraction}. The subset $V:=\mathcal U\cap U$ is an open subset of $S^1\cdot\gamma$ in $\Lambda_kM$.  We define the smooth maps 
\begin{align*}
\Psi_-:[-\tfrac1k,0]\times V\to\Lambda_kM,
\qquad
\Psi_+:[0,\tfrac1k]\times V\to\Lambda_kM,
\end{align*}
by $\Psi_\pm(s,\zeta)=r(e^{2\pi i s}\cdot\zeta)$. These maps glue together to form a continuous (but not smooth) map 
\[\Psi:[-\tfrac1k,\tfrac1k]\times V\to\Lambda_kM.\]
Since the circle action on $\Lambda M$ leaves the energy function invariant, and the map $r$ does not increase the energy, we have $E(\Psi(s,\zeta))\leq E(\zeta)$ with equality if and only if $s\in\{0,1/k,-1/k\}$ or $\zeta$ is a closed geodesic.

Let $X$ be any smooth vector field on $\Lambda_kM$ such that, on the critical circle $S^1\cdot\gamma$, is given by 
\begin{align*}
X(e^{2\pi i s}\cdot\gamma)=\dot\gamma(s+\cdot)=\tfrac{\diff}{\diff s} e^{2\pi is}\cdot\gamma,\qquad
\forall e^{2\pi i s}\in S^1.
\end{align*}
Let $\phi_t$ be the its flow. After shrinking $\Sigma$ around $\gamma$, we can assume that $\overline \Sigma$ is a compact hypersurface with boundary, and $X$ is transverse to $\overline\Sigma$.  Therefore, there exists $\epsilon\in(0,\tfrac1{2k})$ small enough so that the map
\begin{align*}
\Phi:(-\epsilon,\epsilon)\times\Sigma\to N,
\qquad
\Phi(t,\zeta)=\phi_t(\zeta)
\end{align*}
is a diffeomorphism onto a tubular neighborhood $N$ of $\Sigma$. For each $s\in(-\epsilon,\epsilon)$, we set 
\[\Sigma_s:=\Phi(\{s\}\times\Sigma)=\phi_s(\Sigma).\]
Notice that $\Sigma_0=\Sigma$, and more generally $\Sigma_s$ is a hypersurface intersecting the critical circle $S^1\cdot\gamma$ transversely at $s\cdot\gamma$. The hypersurface $\Sigma$ separates $N$ in two connected components. We denote by $N_+$ the connected component of $N\setminus\Sigma$ containing $e^{2\pi i s}\cdot\gamma$ for $s\in(0,\epsilon)$, and $N_-$ the other one. Notice that 
\begin{align*}
\Psi(r,e^{2\pi i s}\cdot\gamma)
=
\phi_r(e^{2\pi i s}\cdot\gamma)=e^{2\pi i (r+s)}\cdot\gamma,
\qquad
\forall r,r+s\in[-\epsilon,\epsilon].
\end{align*}
In particular, there exists a neighborhood $W\subset N\cap V$ of $\gamma$ such that
\begin{itemize}
\item for all $\zeta\in W\cap N_+$ and $r\in[0,\tfrac1k]$, if $\Psi_+(r,\zeta)\in\Sigma_s$ then $\tfrac{\partial}{\partial r} \Psi_+(r,\zeta)$ is transverse to $\Sigma_s$; analogously, for all $\zeta\in W\cap N_-$ and $r\in[-\tfrac1k,0]$, if $\Psi_-(r,\zeta)\in\Sigma_s$ then $\tfrac{\partial}{\partial r} \Psi_-(r,\zeta)$ is transverse to $\Sigma_s$;

\item for all $\zeta\in W\cap N_\pm$ there exists $\pm\delta\in(0,\epsilon)$ such that $\Psi(\mp\delta,\zeta)\in N_{\mp}$.
\end{itemize}

By the implicit function theorem, there exist smooth functions
\begin{align*}
\tau_+:W\cap N_+\to(-\epsilon,0),
\qquad
\tau_-:W\cap N_-\to(0,\epsilon)
\end{align*}
such that $\Psi_{\pm}(\tau_{\pm}(\zeta),\zeta)\in\Sigma$ for all $\zeta\in W\cap N_\pm$. Notice that $\tau_{\pm}(\zeta_n)\to 0$ as $\zeta_n$ converges to an element of $\Sigma$. Therefore, the functions $\tau_+$ and $\tau_-$ glue together to form a continuous function $\tau:W\to(-\epsilon,\epsilon)$ such that $\Psi(\tau(\zeta),\zeta)\in\Sigma$ for all $\zeta\in W$. We define the homotopy $h_t:W\to N$, for $t\in[0,1]$, by $h_t(\zeta):=\Psi(t\tau(\zeta),\zeta)$, so that $h_0=\mathrm{id}$, $h_1(W)\subset\Sigma$, and $h_t$ fixes $\Sigma\cap W$ for all $t\in[0,1]$. Moreover, $E\circ h_t\leq E$ for all $t\in[0,1]$, and therefore $h_t$ preserves the sublevel sets of $E$. This homotopy, together with Proposition~\ref{p:homotopic_invariance_local_homology}, implies point~(i).

Let $K\subset S^1\cdot\gamma$ be an embedded open interval containing $\gamma$. In particular, the closure of $K$ in $\Lambda M$ is homeomorphic to a compact interval. There is a unique smooth function $\tau:\overline K\to(-1,1)$ such that $\tau(\gamma)=0$ and, more generally, $e^{2\pi i\tau(\zeta)}\cdot\zeta =\gamma$ for all $\zeta\in K$. By Tietze's extension Theorem, we can extend $\tau$ to a continuous function $\tau:\Lambda M\to\R$. We define a continuous homotopy $h_s:\Lambda M\to\Lambda M$, for $s\in[0,1]$, by $h_s(\zeta)=e^{2\pi i s \tau(\zeta)}\cdot\zeta$. Let $c>0$ be the critical value of $K$, that is, $K\subset E^{-1}(c)$. Notice that $h_0$ is the identity, every $h_s$ preserves both the sublevel set $\{E<c\}$ and the union $\{E<c\}\cup K$, and $h_1$ maps $\{E<c\}\cup K$ to $\{E<c\}\cup \{\gamma\}$. In particular, $h_1$ is a homotopic inverse of the inclusion 
\begin{align*}
(\{E<c\}\cup\{\gamma\},\{E<c\}) \subset (\{E<c\}\cup K,\{E<c\}),
\end{align*}
and point (ii) follows.

The action by a fixed element $\theta=e^{2\pi i t}\in S^1$ gives a diffeomorphism of $\Lambda M$ that restricts to a homeomorphism of pairs
\begin{align}\label{e:action_by_theta}
(\{E<c\}\cup\{\gamma\},\{E<c\}) \toup^{\cong} (\{E<c\}\cup\{\theta\cdot\gamma\},\{E<c\}).
\end{align}
In particular, the action by $\theta$ induces an isomorphism in local homology as claimed in point~(iii). By composing the homeomorphism~\eqref{e:action_by_theta} with an inclusion we obtain a continuous map of pairs
\begin{align*}
\Theta:(\{E<c\}\cup\{\gamma\},\{E<c\}) \to (\{E<c\}\cup K,\{E<c\}),\qquad \Theta(\zeta)=\theta\cdot\zeta,
\end{align*}
which induces an isomorphism in homology according to point~(ii). We define a continuous homotopy $\Theta_s:\Lambda M\to\Lambda M$, $s\in[0,1]$, by $\Theta_s(\zeta)=e^{2\pi i t s}\cdot\zeta$. If $K$ contains $e^{2\pi i t s}\cdot\gamma$ for all $s\in[0,1]$, $\Theta_s$ restricts to a homotopy of pairs of the form
\begin{align*}
\Theta_s:(\{E<c\}\cup\{\gamma\},\{E<c\}) \to (\{E<c\}\cup K,\{E<c\}).
\end{align*}
This proves the commutativity of the diagram in point~(iii).
\end{proof}

In the proof of Theorem~\ref{t:main}, we will need the following statement concerning closed geodesics with particularly simple local homology. Its proof can also be extracted from the one of \cite[Proposition~1]{Hingston:1993ou}.

\begin{lem}\label{l:generator_local_homology}
Let $S^1\cdot\gamma\in\crit(E)\cap\Lambda_kM$ be the isolated critical circle of a closed geodesic $\gamma$ whose local homology $\Loc_d(E,\gamma)$ is non-trivial in degree $d=\ind(\gamma)+\nul(\gamma)$. Then, the local homology $\Loc_*(E,\gamma)$ is concentrated in degree $d$. A generator of $\Loc_d(E,\gamma)$ is given by any smoothly embedded $d$-dimensional compact ball $B^d\subset\Lambda_k M$ intersecting $S^1\cdot\gamma$ in its interior at $\gamma$ with a non-tangent intersection and such that $E|_{B^d\setminus\{\gamma\}}<E(\gamma)$.
\end{lem}

\begin{proof}
Let $\Sigma\subset\Lambda_kM$ be a hypersurface transverse to the critical circle $S^1\cdot\gamma$ and containing the ball $B^d$. By Proposition~\ref{p:isolated}, since $S^1\cdot\gamma$ is isolated in $\crit(E)$, $\gamma$ is an isolated critical point of the restricted energy $E|_{\Sigma}$. By Lemma~\ref{l:local_homology_and_shift}(i), the inclusion induces a local homology isomorphism
\begin{align*}
\Loc_*(E|_{\Sigma},\gamma)
\ttoup_{\cong}^{\incl}
\Loc_*(E,\gamma).
\end{align*}
Since the tangent space $\Tan_{\gamma}(S^1\cdot\gamma)$ is contained in the kernel of the Hessian of $E$ at $\gamma$, we have that $\ind(E,\gamma)=\ind(E|_{\Sigma},\gamma)$ and $\nul(E,\gamma)=\nul(E|_{\Sigma},\gamma)$. Therefore, the local homology $\Loc_d(E|_{\Sigma},\gamma)$ is non-trivial in degree $d=\ind(E|_{\Sigma},\gamma)+\nul(E|_{\Sigma},\gamma)$. By a general Morse theoretic statement (see \cite[page~256]{Hingston:1993ou} or \cite[Prop.~2.6(ii)]{Mazzucchelli:2013co}), the local homology $\Loc_*(E|_{\Sigma},\gamma)$ is concentrated in degree $d$, and is generated by $B^d$. Thus, the same is true for the local homology $\Loc_*(E,\gamma)$ as well.
\end{proof}

The circle action on $\Lambda M$ is not free. Indeed, given a loop $\zeta\in\Lambda M$ and an integer $m\in\N$, we denote by $\zeta^m\in\Lambda M$ the $m$-th iterate of $\zeta$, which is defined by $\zeta^m(t)=\zeta(mt)$. If a loop $\gamma\in\Lambda M$ is not a stationary curve, we denote by $\mul(\gamma)$ its multiplicity as an iterated curve, that is, the maximal $m\in\N$ such that $\gamma=\zeta^m$ for some $\zeta\in\Lambda M$. Let $\Z_m\subset S^1$ be the cyclic subgroup of order $m=\mul(\gamma)$ generated by $\mu=e^{2\pi i/m}$. The isotropy group of $\gamma$ is precisely $\Z_m$, meaning
\begin{align*}
\Z_m=\{ \theta\in S^1\ |\ \theta\cdot\gamma=\gamma \}.
\end{align*}

The following statement, which is a slightly more detailed version of the first part of \cite[Prop.~3.2]{Bangert:2010ak}, describes the relation between the local homology of a critical circle and the local homology of a critical point on it. A closed geodesic $\gamma$ is said to be non-iterated when $\mul(\gamma)=1$.

\begin{lem}
\label{l:exact_sequence}
Let $\gamma\in\crit(E)$ be a non-iterated closed geodesic. Consider an integer $m\in\N$ and the generator $\mu=e^{2\pi i/m}$ of the subgroup $\Z_m\subset S^1$. For each $m\in\N$ and $d\in\N\cup\{0\}$, we have
\begin{align*}
\Loc_d(E,S^1\cdot\gamma^m)
\cong
\Loc_d(E,\gamma^m)^{\Z_m} \oplus \Loc_{d-1}(E,\gamma^m)^{\Z_m}.
\end{align*}
Moreover, the  sequence 
\begin{align*}
0
\toup (\mu_*-\mathrm{id})\Loc_*(E,\gamma^m)
\ttoup^{\incl}
 \Loc_*(E,\gamma^m)
\ttoup^\incl \Loc_*(E,S^1\cdot\gamma^m)
\end{align*}
is exact.
\end{lem}

\begin{proof}
For any $\epsilon\in(0,\tfrac 1{4m})$, we set
\begin{align*}
I_1 & :=\big\{ e^{2\pi i s}\ \big|\ s\in(-\epsilon,\tfrac1{2m}+\epsilon)  \big\},\\
I_2 & :=\big\{ e^{2\pi i s}\ \big|\ s\in(\tfrac1{2m}-\epsilon,\tfrac1m+\epsilon)  \big\},\\
J_1 & :=\big\{ e^{2\pi i s}\ \big|\ s\in(-\epsilon,\epsilon)  \big\},\\
J_2 & :=\big\{ e^{2\pi i s}\ \big|\ s\in(\tfrac1{2m}-\epsilon,\tfrac1{2m}+\epsilon)  \big\},
\end{align*}
so that 
\begin{align*}
 (I_1\cdot\gamma^m) \cup (I_2\cdot\gamma^m) & = S^1\cdot\gamma^m,\\
 (I_1\cdot\gamma^m) \cap (I_2\cdot\gamma^m) & =(J_1\cdot\gamma^m)\sqcup(J_2\cdot\gamma^m),
\end{align*}
see Figure~\ref{f:circle}. 
\begin{figure}
\begin{center}
\begin{footnotesize}
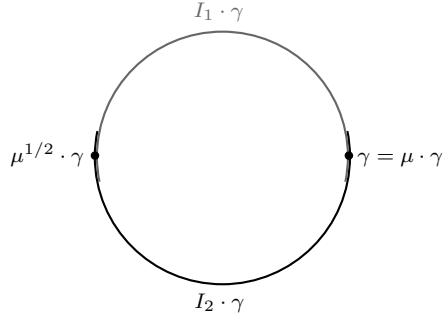 
\end{footnotesize}
\begin{small}
\caption{Decomposition of the critical circle $S^1\cdot\gamma$ as the union $(I_1\cdot\gamma)\cup (I_2\cdot\gamma)$.}
\label{f:circle}
\end{small}
\end{center}
\end{figure}%
We set $c:=E(\gamma^m)$. The relative Mayer-Vietoris sequence associated to decomposition of $(\{E<c\}\cup S^1\cdot\gamma,\{E<c\})$ as a union
\begin{align*}
(\{E<c\}\cup I_1\cdot\gamma^m,\{E<c\}) 
\cup
(\{E<c\}\cup I_2\cdot\gamma^m,\{E<c\}) 
\end{align*}
reads
\begin{align*}
\xymatrix{
\Loc_*(E,J_1\cdot\gamma^m) \oplus \Loc_*(E,J_2\cdot\gamma^m)
\ar[r]^{\alpha_*}
&
\Loc_*(E,I_1\cdot\gamma^m) \oplus \Loc_*(E,I_2\cdot\gamma^m)
\ar[d]^{\beta_*}
\\
&\Loc_*(E,S^1\cdot\gamma^m)
\ar[ul]^{\partial_*}
} 
\end{align*}
where
\begin{align*}
\alpha_*(h_1,h_2)&=(h_1+h_2,h_1+h_2) ,\\
\beta_*(h_1, h_2)&=h_1-h_2,
\end{align*}
and the connecting homomorphism $\partial_*$ lowers the grading $*$ by one. In these expressions, for simplicity, we left the compositions with the inclusion-induced homomorphisms implicit. Since the Mayer-Vietoris sequence is exact and we work with rational coefficients for the homology, we have
\begin{align*}
C_*(E,S^1\cdot\gamma^m)\cong\ker(\alpha_{*-1})\oplus\textrm{im}(\beta_*).
\end{align*}

Notice that $\mu^{1/2}=e^{\pi i/m}\in J_2$. There is a commutative diagram
\begin{align*}
\xymatrix{
\Loc_*(E,\gamma^m) \oplus \Loc_*(E,\mu^{1/2}\cdot\gamma^m)
\ar[r]^{A_*} \ar[d]_{\incl}^{\cong}
&
\Loc_*(E,\gamma^m) \oplus \Loc_*(E,\gamma^m) \ar[d]_{\cong}^{\incl}\\
\Loc_*(E,J_1\cdot\gamma^m) \oplus \Loc_*(E,J_2\cdot\gamma^m)
\ar[r]^{\alpha_*} 
&
\Loc_*(E,I_1\cdot\gamma^m) \oplus \Loc_*(E,I_2\cdot\gamma^m)
} 
\end{align*}
where the vertical homomorphisms induced by the inclusion are isomorphisms (according to Lemma~\ref{l:local_homology_and_shift}(iii)), and
\begin{align*}
A_*(h_1,h_2)=(h_1+(\mu^{-1/2})_*h_2,h_1+(\mu^{1/2})_*h_2).
\end{align*}
The kernel of $\alpha_*$ is isomorphic to the kernel of $A_*$, which is precisely
\begin{align*}
\ker(A_*) & =\big\{(h_1,h_2)\in\Loc_*(E,\gamma^m) \oplus \Loc_*(E,\mu^{1/2}\cdot\gamma^m) \ \big|\ \mu_*h_1=-(\mu^{1/2})_*h_2=h_1\big\}\\
& \cong \ker\big((\mu_*-\textrm{id}):\Loc_*(E,\gamma^m)\to\Loc_*(E,\gamma^m)\big)\\
& = \Loc_*(E,\gamma^m)^{\Z_m}.
\end{align*}
The image of $A_*$ is given by
\begin{align*}
\textrm{im}(A_*)=\Delta_* + \textrm{graph}(\mu_*),
\end{align*}
where 
\begin{align*}
\Delta_* & =\big\{(h,h)\in\Loc_*(E,\gamma^m)\oplus\Loc_*(E,\gamma^m)\big\},\\
\textrm{graph}(\mu_*) & = \big\{(h,\mu_*h)\ \big|\ h\in\Loc_*(E,\gamma^m)\big\}.
\end{align*}
Notice that
\begin{align*}
\Delta_*\cap\textrm{graph}(\mu_*) 
\cong \ker\big((\mu_*-\textrm{id}):\Loc_*(E,\gamma^m)\to\Loc_*(E,\gamma^m)\big)
= \Loc_*(E,\gamma^m)^{\Z_m},
\end{align*}
which implies that 
\begin{align*}
\dim(\Loc_*(E,\gamma^m)\oplus\Loc_*(E,\gamma^m)) - \dim(\textrm{im}(A_*)) = \dim(\Loc_*(E,\gamma^m)^{\Z_m}).
\end{align*}
Since the Mayer-Vietoris sequence is exact, this latter equality implies
\[\textrm{im}(\beta_*)\cong
\frac{\Loc_*(E,I_1\cdot\gamma^m)\oplus\Loc_*(E,I_2\cdot\gamma^m)}{\mathrm{im}(\alpha_*)}
\cong\Loc_*(E,\gamma^m)^{\Z_m}.\]
Finally, consider the surjective homomorphism 
\[\iota_*:\Loc_*(E,\gamma^m)\oplus\Loc_*(E,\gamma^m)\to\Loc_*(E,\gamma^m)\]
given by $\iota_*(h_1,h_2)=h_1-h_2$. The commutative diagram
\begin{align*}
\xymatrix{
\Loc_*(E,\gamma^m) \oplus \Loc_*(E,\gamma^m)
\ar[r]^{\quad\qquad\iota_*} \ar[d]_{\incl}^{\cong}
&
\Loc_*(E,\gamma^m) \ar[d]^{\incl}\\
\Loc_*(E,I_1\cdot\gamma^m) \oplus \Loc_*(E,I_2\cdot\gamma^m)
\ar[r]^{\qquad\quad\beta_*} 
&
\Loc_*(E,S^1\cdot\gamma^m)
} 
\end{align*}
together with the above Mayer-Vietoris sequence, implies that the kernel of the inclusion-induced homomorphism 
\begin{align*}
\Loc_*(E,\gamma^m)\ttoup^{\incl}\Loc_*(E,S^1\cdot\gamma^m)
\end{align*}
is precisely
\[
\iota_*(\mathrm{im}(A_*))=\iota_*(\Delta_*+\textrm{graph}(\mu_*))
=\iota_*(\textrm{graph}(\mu_*))=(\mu_*-\mathrm{id})\Loc_*(E,\gamma^m).
\qedhere
\]
\end{proof}

\begin{lem}
\label{l:mu_on_loc}
Let $\gamma\in\crit(E)$ be a non-iterated isolated closed geodesic. Let $m\in\N$ be such that $\nul(\gamma^m)=\nul(\gamma)$, and let $\mu=e^{2\pi i/m}$ be a generator of the subgroup $\Z_m\subset S^1$. The homomorphism $\mu_*:\Loc_*(E,\gamma^m)\to\Loc_*(E,\gamma^m)$ is equal to $(-1)^{\ind(\gamma^m)-\ind(\gamma)}\mathrm{id}$.
\end{lem}

\begin{proof}
Let $k\in\N$ be large enough so that $\gamma\in\Lambda_kM$. Notice that the iterate $\gamma^m$  belongs to $\Lambda_{km}M$. The space of closed broken geodesics $\Lambda_{km}M$ inherits a Riemannian metric from $(M,g)$, which we denote by $G$. More specifically, we have
\begin{align*}
G(\xi,\eta):=\sum_{i=0}^{km-1} g(\xi(\tfrac{i}{km}),\eta(\tfrac{i}{km})),
\qquad
\forall \xi,\eta\in\Tan_{\gamma^m}\Lambda_{km}M.
\end{align*}
Let $\E\subset\Tan_{\gamma^m}\Lambda_{mk}M$ be the orthogonal complement of $\mathrm{span}\{\dot\gamma\}$. Since the subgroup $\Z_m\subset S^1$ acts on $\Lambda_{km}M$ by isometries, $\E$ is $\Z_m$-invariant: for each $\theta\in\Z_m$, its differential at $\gamma^m$ restricts to an isometry $\diff\theta(\gamma^m):\E\to \E$. We denote by ``$\exp$'' the exponential map associated to $G$. Let $B_r\subset \E$ be the open ball of radius $r$ centered at the origin, and notice that $B_r$ is invariant by the isometry $\diff\theta(\gamma^m)$, for each $\theta\in\Z_m$. We choose the radius $r>0$ small enough, so that $\Sigma:=\exp_{\gamma^m}(B_r)\subset\Lambda_{mk}M$ is a well defined smooth hypersurface intersecting the critical circle $S^1\cdot\gamma^m$ only at $\gamma^m$ with a transverse (and even normal) intersection. Since $\theta\circ\exp_{\gamma^m}=\exp_{\gamma^m}\circ\diff\theta(\gamma^m)$, any $\theta\in\Z_m$ restricts to an isometry of $\Sigma$. By Lemma~\ref{l:local_homology_and_shift}(i), it is enough to show that the generator $\mu\in\Z_m$ induces the local homology isomorphism
\begin{align*}
 \mu_*=(-1)^{\ind(\gamma^m)-\ind(\gamma)}\mathrm{id}:\Loc_*(E|_{\Sigma},\gamma^m)\to\Loc_*(E|_{\Sigma},\gamma^m).
\end{align*}
Let $H:\Tan_{\gamma^m}\Lambda_{mk}M\to\Tan_{\gamma^m}\Lambda_{mk}M$ be the self-adjoint linear map associated to the Hessian $\diff^2E(\gamma^m)$, that is, $\diff^2E(\gamma^m)[\xi,\eta]=G(H\xi,\eta)$. The tangent space $\Tan_{\gamma^m}\Lambda_{mk}M$ splits as the direct sum
\begin{align}
\label{e:splitting}
 \Tan_{\gamma^m}\Lambda_{mk}M
 =
 \mathrm{span}\{\dot\gamma\}\oplus\E^0\oplus\E^-\oplus\E^+,
\end{align}
where $\E=\E^0\oplus\E^-\oplus\E^+$, $\ker(H)=\mathrm{span}\{\dot\gamma\}\oplus\E^0$, $\E^-$ is the negative eigenspace of $H$, and $\E^+$ is its positive eigenspace. Since the energy function $E$ is $\Z_m$-invariant, the differential $\diff\mu(\gamma^m)$ preserves the splitting~\eqref{e:splitting}. The eigenvalues of the restrictions $\diff\mu(\gamma^m)|_{\E_-}$, repeated according to their multiplicities, are
\begin{align*}
\underbrace{1,...,1}_{n_-},\underbrace{-1,...,-1}_{n_+},\lambda_1,\overline\lambda_1,...,\lambda_j,\overline\lambda_j,
\end{align*}
where the $\lambda_i$'s are complex $m$-th roots of unity different from $\pm1$. A vector $\xi\in\E_-$ satisfies $\diff\mu(\gamma^m)\xi=\xi$ if and only if it has the form $\xi=\eta^m$ for some vector $\eta$ that is in the negative eigenspace of the operator associated to $\diff^2E(\gamma)$. This shows that $n_-=\ind(\gamma)$. Since $\ind(\gamma^m)=\dim\E_-=n_-+n_++2j$, we have
\begin{align}
\label{e:det}
 \det\diff\mu(\gamma^m)|_{\E_-}
 =
 (-1)^{n_+}
 =
 (-1)^{\ind(\gamma^m)-\ind(\gamma)}.
\end{align}
By the generalized Morse Lemma \cite[Lemma~1]{Gromoll:1969jy}, there exist convex neighborhoods of the origin 
$U_0\subset\E^0\oplus\E^-$ and $U_1\subset\E^+$, 
and a diffeomorphism onto its image $\psi:U:=U_0\times U_1\to\Sigma$ such that $\psi(\bm0)=\gamma$, $\diff\psi(\bm0)=\mathrm{id}$, and $F:=E\circ\psi$ has the form 
\begin{align*}
F(\xi_0,\xi_-,\xi_+)
&=f(\xi_0)+\tfrac12\diff^2E(\gamma^m)[(\xi_-,\xi_+),(\xi_-,\xi_+)]\\ 
&=f(\xi_0)
-\tfrac12|G(H\xi_-,\xi_-)|^2
+\tfrac12|G(H\xi_+,\xi_+)|^2,
\end{align*}
for some smooth function $f$. The homotopy 
\[r_s:U\to U,
\quad
r_s(\xi_0,\xi_-,\xi_+)=(\xi_0,\xi_-,(1-s)\xi_+),
\qquad s\in[0,1],\] is a deformation retraction onto $U_0$ such that $F\circ r_s\leq F$. Its time-1 map $r_1$ restricts to the homotopy inverse of the inclusion of pairs
\[
\big(\{F|_{U_0}<F(\bm0)\}\cup\{\bm0\},\{F|_{U_0}<F(\bm0)\}\big)
\subseteq
\big(\{F<F(\bm0)\}\cup\{\bm0\},\{F<F(\bm0)\}\big),
\]
which therefore induces the local homology isomorphism
\begin{align*}
 \Loc_*(F|_{U_0},\bm0)
 \ttoup_{\cong}^{\incl}
 \Loc_*(F,\bm0).
\end{align*}
Let $V_0\subset U_0$ be an open neighborhood of the origin that is sufficiently small so that $\psi^{-1}\circ\mu\circ\psi(V_0)\subset U$. Let $\phi_s:=r_s\circ\psi^{-1}\circ\mu\circ\psi$, and notice that $F\circ \phi_s\leq F$. In particular, all the $\phi_s$ induce the same local homology isomomorphism
\begin{align*}
(\phi_s)_*=(\phi_0)_*:\Loc_*(F|_{V_0},\bm0)\ttoup^{\cong}\Loc_*(F,\bm0),
\qquad\forall s\in[0,1],
\end{align*}
Since $\phi_1(V_0)\subset U_0$, the homomorphism $(\phi_1)_*$ factorizes as 
\begin{align*}
\xymatrix{
\Loc_*(F|_{V_0},\bm0)
\ar[rr]^{(\phi_1)*}_{\cong} 
\ar[ddr]_{(\phi_1)*} 
& &
\Loc_*(F,\bm0)\\\\
&\Loc_*(F|_{U_0},\gamma^m)
\ar[uur]^{\incl}_{\cong}
& 
} 
\end{align*}
and therefore $\phi:=\phi_1$ also induces the local homology isomorphism
\begin{align*}
 \phi_*:\Loc_*(F|_{V_0},\bm0)\ttoup_{\cong}^{\incl}\Loc_*(F|_{U_0},\bm0).
\end{align*}
Summing up, we have the commutative diagram of local homology isomorphisms
\begin{equation}\label{e:comm_diag}
\xymatrix{
\Loc_*(E,\gamma^m)
\ar[rr]^{\mu_*}_{\cong} 
& &
\Loc_*(E,\gamma^m)\\\\
\Loc_*(F|_{V_0},\gamma^m)
\ar[rr]^{\phi_*}_{\cong} 
\ar[uu]^{\psi_*}_{\cong}
& &
\Loc_*(F|_{U_0},\gamma^m)
\ar[uu]_{\psi_*}^{\cong}
} 
\end{equation}
Since $\diff\psi(\bm0)$ preserves the splitting~\eqref{e:splitting}, we have
\begin{align*}
\diff\phi(\bm0)|_{\E_0}& =\diff\mu(\gamma^m)|_{\E_0}=\mathrm{id},\\
\diff\phi(\bm0)|_{\E_-}& =\diff\mu(\gamma^m)|_{\E_-},
\end{align*}
and thus, by~\eqref{e:det},
\begin{equation}\label{e:det_diff_phi}
\begin{split}
  \mathrm{sign}(\det\diff\phi(\bm0)) & = \mathrm{sign}(\det\diff\mu(\gamma^m)|_{\E_-})\\
  &=\det\diff\mu(\gamma^m)|_{\E_-}\\
  &=(-1)^{\ind(\gamma^m)-\ind(\gamma)}.
\end{split}
\end{equation}
The function $F$ and the map $\phi$ satisfy the assumptions of Proposition~\ref{p:local_homology_maps}, which therefore implies 
\[\phi_*=(\mathrm{sign}(\det\diff\phi(\bm0)))\mathrm{id}=(-1)^{\ind(\gamma^m)-\ind(\gamma)}\mathrm{id}.\] 
This, together the commutative diagram~\eqref{e:comm_diag}, allows to conclude
\[
\mu_*
=\psi_*^{-1}\circ\phi_*\circ\psi_*
=\psi_*^{-1}\circ(\mathrm{sign}(\det\diff\phi(\bm0)))\mathrm{id}\circ\psi_*=(-1)^{\ind(\gamma^m)-\ind(\gamma)}\mathrm{id}.
\qedhere
\]
\end{proof}

The Morse indices of iterated closed geodesics can be computed by means of Bott's iteration theory \cite{Bott:1956sp} (see also~\cite{Long:2002ed, Mazzucchelli:2015zc} for more recent accounts). More specifically, there exists two functions $\Lambda:S^1\to\N\cup\{0\}$ and $N:S^1\to\N\cup\{0\}$ associated to the given closed geodesic $\gamma$, such that for all $m\in\N$ we have
\begin{align}
\label{e:Bott_formula_index}
\ind(\gamma^m) & = \sum_{z\in\!\sqrt[m]{1}} \Lambda(z),\\
\label{e:Bott_formula_nullity}
\nul(\gamma^m) & = \sum_{z\in\!\sqrt[m]{1}} N(z).
\end{align}
If  $\phi_t$ denotes the geodesic flow on $\Tan M$, the function $N$ is given by
\begin{align*}
N(z) := \dim_{\C}\ker_{\C}(\diff\phi_1(\gamma(0),\dot\gamma(0))-z\,\mathrm{id}),
\qquad\forall z\in S^1\subset\C.
\end{align*}
In particular, the support of $N$ is finite. 
The definition of $\Lambda$ is more involved. Here, we just recall that $\Lambda$ is a piecewise constant function with at most finitely many discontinuities contained in the support of $N$. In particular
\begin{align}
\label{e:Bott_average_idx}
 \avind(\gamma)=\lim_{m\to\infty}\frac{\ind(\gamma^m)}{m} = \frac{1}{2\pi} \int_0^{2\pi} \Lambda(e^{i2\pi t})\,\diff t.
\end{align}
Moreover, $\Lambda$ satisfies
\begin{align}
\label{e:Lambda_conj_inv}
 \Lambda(z)=\Lambda(\overline z),
 \qquad
 \forall z\in S^1.
\end{align}
These properties of Bott's functions $\Lambda$ and $N$ allow to draw the following corollary of Lemma~\ref{l:mu_on_loc}.

\begin{cor}
\label{c:incl_local_homologies}
Let $\gamma\in\crit(E)$ be a non-iterated isolated closed geodesic. For all sufficiently large prime numbers $m$, the inclusion induces an injective homomorphism
\begin{align}
\label{e:injective_incl}
 \Loc_*(E,\gamma^m) \eembup^{\incl} \Loc_*(E,S^1\cdot\gamma^m).
\end{align}
\end{cor}

\begin{proof}
If $m$ is a large enough prime number, we have $N(z)=0$ for each $z\in \sqrt[m]{1}$ with $z\neq1$, which implies 
 $\nul(\gamma^m)=\nul(\gamma)$.
Moreover, if $m>2$, in particular it is odd. Therefore, Equation~\eqref{e:Lambda_conj_inv} implies
\begin{align*}
\ind(\gamma^m) & = \ind(\gamma) + 2\!\!\!\sum_{\textstyle
  \substack{    z\in\!\sqrt[m]{1},\\
    \mathrm{Im}(z)>0 }
} \Lambda(z),
\end{align*}
so that $(-1)^{\ind(\gamma^m)-\ind(\gamma)}=1$. This,  together with Lemma~\ref{l:mu_on_loc}, implies that the isomomorphism $\mu_*:\Loc_*(E,\gamma^m)\to\Loc_*(E,\gamma^m)$ is the identity. The exact sequence provided by Lemma~\ref{l:exact_sequence} implies that the inclusion induces the injective homomorphism~\eqref{e:injective_incl}.
\end{proof}

We denote by $\psi^m:\Lambda M\hookrightarrow\Lambda M$ the iteration map $\psi^m(\gamma)=\gamma^m$, which is a smooth embedding of the free loop space into itself. The following lemma will be useful in the proof of Theorem~\ref{t:main} in case the closed geodesic of the statement has zero average index. Its proof can be extracted from the arguments in \cite{Gromoll:1969gh}.

\begin{lem}
\label{l:loc_isom_iteration}
Let $\gamma\in\crit(E)$ be a closed geodesic such that $\avind(\gamma)=0$. For each integer $m\in\N$ such that $\nul(\gamma)=\nul(\gamma^m)$, the iteration map induces the local homology isomorphism
\begin{align*}
 \psi^m : \Loc_*(E,S^1\cdot\gamma) \toup^{\cong}\Loc_*(E,S^1\cdot\gamma^m).
\end{align*}
\end{lem}

\begin{proof}
Since $\avind(\gamma)=0$, Equation~\eqref{e:Bott_average_idx} implies that Bott's function $\Lambda:S^1\to\N\cup\{0\}$ associated to $\gamma$ vanishes identically. By~\eqref{e:Bott_formula_index}, we have
\begin{align}
\label{e:zero_indices}
\ind(\gamma^m)=\sum_{z\in\!\sqrt[m]1} \Lambda(z)=0,\qquad\forall m\in\N.
\end{align}
Let us fix $m\in\N$ such that $\nul(\gamma)=\nul(\gamma^m)$.

It is well known that, if we choose $k\in\N$ large enough, the Morse index and the nullity of $E$ at $S^1\cdot\gamma^{m}$ are the same at those of the restriction $E|_{\Lambda_{km}M}$ at $\gamma$. Since the inclusion induces the local homology isomorphisms
\begin{align*}
\Loc_*(E|_{\Lambda_{k}M},S^1\cdot\gamma)
& \ttoup^{\cong}
\Loc_*(E,S^1\cdot\gamma),
\\
\Loc_*(E|_{\Lambda_{km}M},S^1\cdot\gamma^m)
& \ttoup^{\cong}
\Loc_*(E,S^1\cdot\gamma^m),
\end{align*}
it is enough for us to prove that the iteration map induces the local homology isomorphism
\begin{align*}
\psi^m_*:
\Loc_*(E|_{\Lambda_{k}M},S^1\cdot\gamma)
 \ttoup^{\cong}
 \Loc_*(E|_{\Lambda_{km}M},S^1\cdot\gamma^m).
\end{align*}
From now on, in order to simplify the notation, we will omit to write the restriction to $\Lambda_{km}M$, and thus we will consider $E$ as a function of the form
\begin{align*}
 E:\Lambda_{km}M\to[0,\infty).
\end{align*}

We recall that $\Lambda_{km}M$ is diffeomorphic to an open subset of the $km$-fold product $M\times...\times M$ via the map $\zeta\mapsto(\zeta(0),\zeta(\tfrac {1}{km}),...,\zeta(\tfrac{km-1}{km}))$. The subgroup $\Z_m\subset S^1$ acts on $\Lambda_{km}M$ by isometries with respect of the Riemannian metric
\begin{align*}
G(\xi,\eta):=\sum_{i=0}^{km-1} g(\xi(\tfrac{i}{km}),\eta(\tfrac{i}{km})),
\qquad
\forall \xi,\eta\in\Tan_{\zeta}\Lambda_{km}M.
\end{align*}
Since the energy function $E$ on $\Lambda_{km}M$ has the simple expression
\begin{align*}
E(\zeta) = km \sum_{i=0}^{km-1} \dist(\zeta(\tfrac{i}{km}),\zeta(\tfrac{i+1}{km}))^2,
\qquad\forall \zeta\in\Lambda_{km}M,
\end{align*}
we readily see that 
\begin{align*}
\diff E(\zeta^m)\diff\theta(\zeta^m)=\diff E(\zeta^m),
\qquad
\forall\zeta\in\Lambda_kM,\ \theta\in\Z_m\subset S^1. 
\end{align*}
We denote by $\nabla E$ the gradient of $E$ with respect to the Riemannian metric $G$. For each $\theta\in\Z_m$, $\zeta\in\Lambda_kM$, and $\eta\in\Tan_{\zeta^m}\Lambda_{km}M$, we have
\begin{align*}
G(\diff\theta(\zeta^m)\nabla E(\zeta^m),\eta)
&=
G(\nabla E(\zeta^m),\diff\theta^{-1}(\zeta^m)\eta)\\
&=
\diff E(\zeta^m)\circ \diff\theta^{-1}(\zeta^m)\eta\\
&=
\diff E(\zeta^m)\eta\\
&=
G(\nabla E(\zeta^m),\eta),
\end{align*}
and therefore the differential of the $\Z_m$ action on $\Lambda_{km}M$ leaves $\nabla E$ invariant, i.e.
\begin{align*}
 \diff\theta(\zeta^m)\nabla E(\zeta^m)=\nabla E(\zeta^m),
 \qquad
 \forall\zeta\in\Lambda_kM,\ \theta\in\Z_m.
\end{align*}
Geometrically, this means that  $\nabla E$ is tangent to the submanifold of iterated curves $\psi^m(\Lambda_kM)$, i.e.
\begin{align}
\label{e:gradient_iterated}
 \nabla E(\zeta^m)\in\Tan_{\zeta^m}(\psi^m(\Lambda_kM)).
\end{align}
We denote by $N(\psi^m(\Lambda_kM))$ the normal bundle of $\psi^m(\Lambda_kM)\subset\Lambda_{km}M$, which is the vector bundle over $\psi^m(\Lambda_kM)$ whose fibers are the normal to the tangent spaces of $\psi^m(\Lambda_kM)$. Equation~\eqref{e:gradient_iterated} readily implies that
\begin{align*}
 \diff^2E(\zeta^m)[\xi^m,\eta]=0,
 \qquad
 \forall\xi^m\in\Tan_{\zeta^m}(\psi^m(\Lambda_kM)),\ \eta\in N_{\zeta^m}(\psi^m(\Lambda_kM)).
\end{align*}
This, together with the fact that $\ind(\gamma)=\ind(\gamma^m)=0$ and $\nul(\gamma)=\nul(\gamma^m)$, readily implies that the Hessian $\diff^2E(\tau\cdot\gamma^m)$ is positive definite on $N_{\tau\cdot\gamma^m}(\psi^m(\Lambda_kM))$ for all $\tau\in S^1$. Therefore, for each sufficiently small neighborhood $U\subset \Lambda_kM$ of the critical circle $S^1\cdot\gamma$, there exists a tubular neighborhood $W\subset \Lambda_{km}M$ of $\psi^m(U)$ that we can identify with an open neighborhood of the zero section of the normal bundle $N(\psi^m(U))$, such that the restriction of $E$ to any fiber of $W$ has a non-degenerate local minimum at the intersection with the zero section. Up to shrinking $W$, the radial contractions give a deformation retraction $r_s:W\to W$ such that $r_0=\mathrm{id}$, $r_1:W\to\psi^m(U)$, and $E\circ r_s\leq E$ for each $s\in[0,1]$. In particular, the time-1 map $r_1$ is a homotopy inverse for the inclusions
\begin{align*}
\psi^m(U)\cap\{E<E(\gamma^m)\}& \ \  \subset \ \ W\cap\{E<E(\gamma^m)\}, \\
\psi^m(U)\cap\{E<E(\gamma^m)\}\cup S^1\cdot\gamma^m& \ \ \subset \ \ W\cap\{E<E(\gamma^m)\}\cup S^1\cdot\gamma^m.
\end{align*}
Therefore, the inclusion induces the local homology isomorphism
\begin{align*}
\Loc_*(E|_{\psi^m(U)},S^1\cdot\gamma^m) \ttoup^{\incl}_{\cong}
\Loc_*(E|_{W},S^1\cdot\gamma^m),
\end{align*}
which is equivalent to say that the iteration map induces the 
local homology isomorphism
\[
\psi^m_*:\Loc_*(E,S^1\cdot\gamma) \ttoup^{\cong}
\Loc_*(E,S^1\cdot\gamma^m).
\qedhere
\]
\end{proof}

\section{The Morse index in the based loop space}
\label{s:idx_based}

Let $\gamma:\R/\Z\to M$ be a closed geodesic in a Riemannian manifold $(M,g)$. We set $q_*:=\gamma(0)$, and consider the subspace $\Omega M\subset\Lambda M$ of loops in $M$ based at $q_*$, that is
\begin{align}\label{e:based_loop_space}
\Omega M:=\big\{ \zeta\in\Lambda M\ |\ \zeta(0)=q_* \big\}.
\end{align}
So far we have considered the Morse index $\ind(\gamma)$ and the nullity $\nul(\gamma)$ of the energy function $E$ at $S^1\cdot\gamma$. 
The closed geodesic $\gamma$ is in particular a critical point of the restricted energy $E|_{\Omega M}$, and we denote by $\ind_\Omega(\gamma)$ and $\nul_\Omega(\gamma)$ the corresponding index and nullities, that is,
\begin{align*}
\ind_\Omega(\gamma) & := \max\big\{ \dim\E\ \big|\ \E\subset\Tan_\gamma\Omega M\mbox{ with }\diff^2 E(\gamma)[\xi,\xi]<0\ \forall\xi\in\E\setminus\{0\}\big\},\\
\nul_\Omega(\gamma) & := \dim\ker(\diff^2 E|_{\Omega M}(\gamma)).
\end{align*}
Notice that
\begin{align*}
 \ind_\Omega(\gamma) +\nul_\Omega(\gamma) = \max\big\{ \dim\E\ \big|\ \E\subset\Tan_\gamma\Omega M\mbox{ with }\diff^2 E(\gamma)[\xi,\xi]\leq 0\ \forall\xi\in\E\big\}.
\end{align*}
It readily follows from the definition of the Morse index that $\ind_{\Omega}(\gamma)\leq\ind(\gamma)$. Moreover, 
\begin{align}\label{e:Ball_Thor_Zill}
\ind_\Omega(\gamma)+\nul_\Omega(\gamma)
\leq
\ind(\gamma)+\nul(\gamma)
\leq
\ind_\Omega(\gamma)+\nul_\Omega(\gamma)+\dim(M)-1,
\end{align}
see \cite[Eq.~(1.7)]{Ballmann:1982rz}. As for the behavior of the $\Omega$-Morse indices under iteration, we have the following inequalities, which were remarked in \cite[pages 256-257]{Hingston:1993ou}.
\begin{lem}\label{l:iteration_Omega_index}
For all $m\in\N$,
\begin{align*}
\ind_\Omega(\gamma^m) & \geq m \, \ind_\Omega(\gamma),\\
 \ind_\Omega(\gamma^m)+\nul_\Omega(\gamma^m) & \geq m \big( \ind_\Omega(\gamma) + \nul_\Omega(\gamma)\big).
\end{align*}
\end{lem}
\begin{proof}
For each $[t_1,t_2]\in\R$, we introduce the path space
\begin{align*}
\Omega_{t_1,t_2} := \big\{ \zeta\in W^{1,2}([t_1,t_2],M)\ \big|\ \zeta(t_1)=\gamma(t_1),\ \zeta(t_2)=\gamma(t_2)\big\},
\end{align*}
so that $\Omega M=\Omega_{0,1}$. We denote by $E_{t_1,t_2}:\Omega_{t_1,t_2}\to\R$ the energy functional over this space, that is
\begin{align*}
E_{t_1,t_2}(\zeta) = \int_{t_1}^{t_2} g(\dot\zeta(t),\dot\zeta(t))\,\diff t.
\end{align*}
Consider three values $t_1<t_2<t_3$. Notice that  $\Tan_\gamma\Omega_{t_1,t_2}$ and $\Tan_\gamma\Omega_{t_2,t_3}$ can be seen as the vector subspaces of $\Tan_\gamma\Omega_{t_1,t_3}$ given by those vector fields with support in $[t_1,t_2]$ and $[t_2,t_3]$ respectively, and we have
\begin{align*}
\diff^2E_{t_1,t_3}(\gamma)[\xi,\eta]=0,
\qquad
\forall \xi\in\Tan_\gamma\Omega_{t_1,t_2},\ \eta\in\Tan_\gamma\Omega_{t_2,t_3}.
\end{align*}
This readily implies 
\begin{align*}
\ind(E_{t_1,t_3},\gamma) & \geq \sum_{i=1,2} \ind(E_{t_i,t_{i+1}},\gamma),\\
\ind(E_{t_1,t_3},\gamma)+\nul(E_{t_1,t_3},\gamma) & \geq \sum_{i=1,2} \Big( \ind(E_{t_i,t_{i+1}},\gamma)+\nul(E_{t_i,t_{i+1}},\gamma) \Big).
\end{align*}
Moreover, since $\gamma$ is 1-periodic, we have 
\begin{align*}
 \ind(E_{t_1,t_2},\gamma)=\ind(E_{t_1+k,t_2+k},\gamma),
 \qquad
 \forall k\in\Z,\ t_1<t_2.
\end{align*}
This, together with 
\begin{align*}
 \ind_\Omega(\gamma) & = \ind(E_{0,1},\gamma),\\
 \ind_\Omega(\gamma^m) & = \ind(E_{0,m},\gamma),\\
 \ind_\Omega(\gamma)+\nul_\Omega(\gamma) & = \ind(E_{0,1},\gamma)+\nul(E_{0,1},\gamma),\\ 
 \ind_\Omega(\gamma^m)+\nul_\Omega(\gamma^m) & = \ind(E_{0,m},\gamma)+\nul(E_{0,m},\gamma), 
\end{align*}
implies our desired inequalities.
\end{proof}

\section{The case of positive average index}
\label{s:avind_positive}

\subsection{Construction of local homology generators}
\label{ss:construction}

Let $(M,g)$ be a complete Riemannian manifold of dimension at least 2 without close conjugate points at infinity. As usual, we denote by $E:\Lambda M\to[0,\infty)$ its energy function. Since we will be looking for infinitely many closed geodesics, we can assume that $\crit(E)\cap E^{-1}(0,\infty)$ is a collection of isolated critical circles.
Let $\gamma\in\crit(E)$ be a closed geodesic satisfying the assumptions of Theorem~\ref{t:main}, and having positive average index
\begin{align}
\label{e:positive_average_index}
 \avind(\gamma)>0.
\end{align}
In this case, the proof of Theorem~\ref{t:main} which we are going to present combines the arguments of Hingston \cite{Hingston:1993ou} with the penalization trick of Benci and Giannoni (see Section~\ref{s:Benci_Giannoni}).

We set $c:=E(\gamma)>0$. We choose $k\in\N$ large enough so that $\gamma\in\Lambda_kM$, and we introduce the evaluation map
\begin{align*}
\ev:\Lambda_kM\to M,
\qquad
\ev(\zeta)=\zeta(0).
\end{align*}
This map is a submersion. Indeed, every closed broken geodesic $\zeta\in\Lambda_k$ is parametrized by the $k$-tuple $(\zeta(0),\zeta(\tfrac1k),...,\zeta(\tfrac{k-1}k))$, and the evaluation map $\ev$ is simply the projection onto the first coordinate.

Let $\Sigma_0\subset M$ be an embedded hypersurface intersecting the closed geodesic $\gamma$  only at the point 
$q_*:=\gamma(0)$,
and such that this intersection is transverse. The preimage 
\[\Sigma:=\ev^{-1}(\Sigma_0)\subset\Lambda_kM\] 
is a hypersurface intersecting the critical circle $S^1\cdot\gamma$ only at $\gamma$. This intersection is transverse, for the vector field $\dot\gamma\in\Tan_\gamma(S^1\cdot\gamma)$ satisfies
\begin{align*}
\diff\ev(\gamma)\dot\gamma = \dot\gamma(0)\not\in \Tan_{q_*}\Sigma_0.
\end{align*}
Up to shrinking $\Sigma_0$ around $q_*$ and $\Sigma$ around $\gamma$, we can assume that any closed broken geodesic $\zeta\in\Sigma$ intersects the hypersurface $\Sigma_0$ only at the point $\zeta(0)$.

By assumptions~(i-ii) in Theorem~\ref{t:main}, the closed geodesic $\gamma$ has non-trivial local homology $\Loc_d(E,\gamma)$ in degree
\begin{align*}
d=\avind(\gamma) + \dim(M)-1 = \ind(\gamma) + \nul(\gamma)>1.
\end{align*}
Lemma~\ref{l:local_homology_and_shift}(i) implies that the inclusion induces the local homology isomorphism \[\Loc_d(E|_\Sigma,\gamma)\ttoup_{\cong}^{\incl}\Loc_d(E,\gamma).\]
Notice that $\gamma$ is an isolated critical point of the functional $E|_\Sigma$ with non-zero local homology in maximal degree $d=\ind(\gamma) + \nul(\gamma)$. By Lemma~\ref{l:generator_local_homology}, the local homology $\Loc_d(E|_\Sigma,\gamma)$ is generated by any embedded $d$-dimensional closed ball $B^d\subset\Sigma$ such that $\gamma$ belongs to the interior of $B^d$ and $E|_{B^d\setminus\{\gamma\}}<c$. We choose such $B^d$ so that its tangent space at $\gamma$ is given by $\Tan_\gamma B^d=\E_-\oplus \E_0$, where $\E_-$ and $\E_0$ are the negative eigenspaces and the kernel of the Hessian $\diff^2E|_{\Sigma}(\gamma)$ respectively.

We denote by $\Omega M$ the based loop space~\eqref{e:based_loop_space}, where the base-point is $q_*=\gamma(0)$. Assumption~(ii) in Theorem~\ref{t:main} forces an identity between the Morse indices of $\gamma$ in $\Omega M$ and in $\Lambda M$.

\begin{lem}\label{l:equality_indices}
$\ind_\Omega(\gamma)+\nul_\Omega(\gamma)=\ind(\gamma)+\nul(\gamma)-(\dim(M)-1)$.
\end{lem}

\begin{proof}
By assumption~(ii) of Theorem~\ref{t:main}, for all $m\in\PP$ we have
\begin{align}
\label{e:index_growth_alternative}
\ind(\gamma^m) + \nul(\gamma^m)
=
m\,(\ind(\gamma)+\nul(\gamma)) - (m-1)(\dim(M)-1).
\end{align}
Since the based loop space $\Omega M$ is a subspace of the free loop space $\Lambda M$, we have
$\ind_\Omega(\gamma^m)\leq\ind(\gamma^m)$ for all $m\in\N$. This, together with Lemma~\ref{l:iteration_Omega_index} and Equations~\eqref{e:Ball_Thor_Zill} and~\eqref{e:index_growth_alternative}, imply that for all $m\in\PP$ we have
\begin{align*}
 \ind_\Omega(\gamma)+\nul_\Omega(\gamma)
 & \leq
 \tfrac{1}{m}\big( \ind_\Omega(\gamma^m) +\nul_\Omega(\gamma^m) \big)\\
 & \leq
 \tfrac{1}{m} \big( \ind(\gamma^m) + \nul(\gamma^m) \big)\\
 & =
 \ind(\gamma) + \nul(\gamma) - \tfrac{m-1}{m}(\dim(M)-1).
\end{align*}
Since $\PP$ is an infinite set of primes, by taking the limit for $m\to\infty, m\in\PP$ we obtain
\begin{align*}
  \ind_\Omega(\gamma)+\nul_\Omega(\gamma)  \leq
 \ind(\gamma) + \nul(\gamma)  - (\dim(M)-1).
\end{align*}
By~\eqref{e:Ball_Thor_Zill}, the opposite inequality holds as well.
\end{proof}

\begin{lem}\label{l:ev_submersion}
$\diff\ev(\gamma)(\Tan_\gamma B^d)=\Tan_{q_*}\Sigma_0$.
\end{lem}

\begin{proof}
Since $\diff\ev(\gamma)\xi=\xi(0)$ and $\ev(B^d)\subset\Sigma_0$, we have that
\[\ker (\diff\ev(\gamma)|_{\Tan_\gamma B^d}) = \Tan_\gamma B^d \cap \Tan_\gamma\Omega M.\]
Therefore, in order to prove the lemma it is enough to show that 
\begin{align*}
\dim(\Tan_\gamma B^d) - \dim(\Tan_\gamma B^d\cap\Tan_\gamma\Omega) \geq \dim(\Tan_{q_*}\Sigma_0),
\end{align*}
that is,
\begin{align*}
\dim(\Tan_\gamma B^d\cap\Tan_\gamma\Omega) \leq \ind(\gamma)+\nul(\gamma) - (\dim(M)-1).
\end{align*}
Notice that $\E_0\cap\Tan_\gamma\Omega M$ is contained in the kernel of the Hessian $\diff^2E|_{\Omega M}(\gamma)$, and therefore 
\begin{align}
\label{e:dim1}
\dim(\E_0\cap\Tan_\gamma\Omega M)\leq\nul_\Omega(\gamma). 
\end{align}
Since $\Tan_\gamma B^d=\E_-\oplus \E_0$, we have that 
\begin{align*}
 \diff E(\gamma)[\xi,\xi]<0,
 \qquad
 \forall\xi\in (\Tan_\gamma B^d\setminus\E_0)\cap\Tan_\gamma\Omega M,
\end{align*}
and therefore
\begin{align}\label{e:dim2}
\dim\left(\frac{\Tan_\gamma B^d\cap\Tan_\gamma\Omega M}{\E_0\cap\Tan_\gamma\Omega M}\right)\leq\ind_\Omega(\gamma).
\end{align}
The inequalities~\eqref{e:dim1} and~\eqref{e:dim2}, together with Lemma~\ref{l:equality_indices}, imply the desired inequality
\[
\dim(\Tan_\gamma B^d\cap\Tan_\gamma\Omega M)
\leq
\ind_\Omega(\gamma) + \nul_\Omega(\gamma)
=
\ind(\gamma)+\nul(\gamma)-(\dim(M)-1).
\qedhere
\]
\end{proof}

Lemma~\ref{l:ev_submersion} implies that the restriction $\ev|_{B^d}:B^d\to\Sigma_0$ is a submersion on a neighborhood of $\gamma$. By the total rank theorem, up to shrinking $\Sigma_0$ around $q_*=\gamma(0)$ and $B^d$ around $\gamma$, there exists an open ball $B'$ of dimension $d-(\dim(M)-1)$ and a diffeomorphism $\phi:\Sigma_0\times B'\to B^d$ such that 
\begin{align*}
 \ev\circ\phi(q,p)=q,
\qquad
 \forall (q,p)\in \Sigma_0\times B'.
\end{align*}
We denote by $p_*\in B'$ the point such that $\phi(q_*,p_*)=\gamma$.

For each $m\in\PP$, we consider the hypersurface
\begin{align*}
 \Sigma_m:=\big\{ \zeta\in\Lambda_{mk}M\ \big|\ \zeta(0)\in\Sigma_0\big\}\subset \Lambda_{mk}M,
\end{align*}
which intersects the critical circle $S^1\cdot\gamma^m$ only at $\gamma^m$ with a transverse intersection. We define the map
\begin{align*}
\Phi_m: \Sigma_0\times \Sigma_0\times (B')^{\times m}\to \Sigma_{m}\subset\Lambda_{mk}M
\end{align*}
by $\Phi_m(q,q',p_0,p_1,...,p_{m-1})=\zeta$, where $\zeta$ is the closed broken geodesic given by
\begin{align*}
 \zeta(\tfrac{j+i}{mk})
 =
 \left\{
   \begin{array}{lll}
    \phi(q,p_j)(\tfrac ik), &  & j=0,...,\lfloor m/2\rfloor-1,\ \ i=0,...,k-1, \\ 
    \phi(q',p_j)(\tfrac ik), &  & j=\lfloor m/2\rfloor,...,m-1,\ \ i=0,...,k-1.
  \end{array}
 \right.
\end{align*}
In other words, if we consider the diffeomorphism onto its image 
\begin{align*}
 \iota_m:\Lambda_{mk}M\to \underbrace{M\times...\times M}_{\times mk},
 \qquad
 \iota_m(\zeta)=(\zeta(0),\zeta(\tfrac 1{mk}),...,\zeta(\tfrac{mk-1}{mk})),
\end{align*}
we have
\begin{align*}
&\iota_m\circ\Phi_m(q,q',p_0,p_1,...,p_{m-1})\\
&\quad=(\iota_1\circ\phi(q,p_0),...,\iota_1\circ\phi(q,p_{\lfloor m/2\rfloor-1}),\iota_1\circ\phi(q',p_{\lfloor m/2\rfloor}),...\iota_1\circ\phi(q',p_{m-1})).
\end{align*}
Since $\iota_1\circ\phi$ is an embedding, $\Phi_m$ is an embedding as well. The energy of the closed broken geodesics in the image of $\Phi_m$ is given by
\begin{equation}
\label{e:energy_Phi_m}
\begin{split}
 m^{-1}\,E\circ \Phi_m(q,q',p_1,...,p_{m-1}) = & \sum_{i=0}^{\lfloor m/2\rfloor-1}E\circ\phi(q,p_i) \\  
 & 
 +\sum_{i=\lfloor m/2\rfloor}^{m-1}E\circ\phi(q',p_i) \\
 & + 
f(q,p_{\lfloor m/2\rfloor-1},q',p_{\lfloor m/2\rfloor}), 
\end{split}
\end{equation}
where $f:\Sigma_0\times B'\times \Sigma_0\times B'\to\R$ is the smooth function 
\begin{align*}
f(q,p,q',p')
=
&\, k
\Big(\dist(\phi(q,p)(\tfrac{k-1}{k}),q')^2 + \dist(\phi(q',p')(\tfrac{k-1}{k}),q)^2 \\
&
\,- \dist(\phi(q,p)(\tfrac{k-1}{k}),q)^2
- \dist(\phi(q',p')(\tfrac{k-1}{k}),q')^2
\Big).
\end{align*}
Notice that 
\begin{align}
\label{e:control_f}
|f(q,p,q',p')|\leq \rho(\dist(q,q')),
\end{align}
where $\rho:[0,\infty)\to[0,\infty)$ is a continuous and monotone increasing function such that $\rho(0)=0$.
We set 
\[
U_m:=\Sigma_0\times \underbrace{B'\times...\times B'}_{\times m},
\] 
and consider the restriction of $\Phi_m$ to the subset where $q=q'$, that is, the embedding
\begin{gather*}
\phi_m:U_m\hookrightarrow\Sigma_m\subset \Lambda_{mk}M,
\\
\phi_m(q,p_0,p_1,...,p_{m-1}):=\Phi_m(q,q,p_0,p_1,...,p_{m-1}).
\end{gather*}

\begin{lem}
The map $\phi_m$ induces the homology isomorphism
\begin{align*}
 (\phi_m)_*: \Hom_{d_m}(U_m,\partial U_m) \toup^{\cong} \Loc_{d_m}(\gamma^m),
\end{align*}
where $d_m:=\ind(\gamma^m)+\nul(\gamma^m)$.
\end{lem}

\begin{proof}
Notice that
\begin{align*}
\phi_m(q,p_0,p_1,...,p_{m-1})(\tfrac{i+t}m)
=
\phi(q,p_i)(t),\qquad
\forall i=0,...,m-1,\ t\in[0,1],
\end{align*}
that is, $\phi_m(q,p_0,p_1,...,p_{m-1})$ is the  concatenation of the loops 
\[\phi(q,p_0)*\phi(q,p_1)*...*\phi_m(q,p_{m-1}).\] Therefore
\begin{align*}
E\circ \phi_m(q,p_0,p_1,...,p_{m-1}) = m\sum_{i=0}^{m-1}E\circ\phi(q,p_i). 
\end{align*}
The dimension of the domain $U_m$ of the embedding $\phi_m$ is
\begin{align*}
 \dim(U_m)
 & =
 \dim(M)-1+m\big(d-(\dim(M)-1)\big)\\
 & = 
 \dim(M)-1 + m\,\avind(\gamma)\\
 & = 
 \ind(\gamma^m)+\nul(\gamma^m)\\
 & = d_m.
\end{align*}
We set 
$\qq_*:=(q_*,p_*,p_*,...,p_*)\in \Sigma_0\times (B')^{\times m}$. Notice that $\phi_m(\qq_*)=\gamma^m$, and $E\circ \phi_m(\qq)<E(\gamma^m)$ for all $\qq\in\Sigma_0\times (B')^{\times m}$ with $\qq\neq\qq_*$. Therefore, by Lemma~\ref{l:generator_local_homology}, the map $\phi_m$ is a generator of the local homology $\Loc_{*}(\gamma^m)=\Loc_{d_m}(\gamma^m)$, that is, it induces a homology isomorphism as claimed.
\end{proof}

\subsection{A cycle deformation}
We now build, for each prime number $m\in\PP$ large enough, an explicit deformation of the relative cycle $\phi_m$ that will push it inside the sublevel set $\{E<m^2c\}$. If $M$ were a closed manifold, this would be enough to infer the existence of a closed geodesic with energy larger than $m^2c$ and smaller than or equal to the maximum of the energy along the deformation. In our non-compact setting, more work is needed in order to derive such a conclusion, and we will take care of it in Subsection~\ref{ss:local_to_global}.

\begin{lem}
\label{l:homotopy_h_ms}
For each $\epsilon>0$ small enough there exists $\overline m_\epsilon>0$ and, for each $m\in\PP$ with $m\geq\overline m_\epsilon$, a homotopy
\begin{align}
\label{e:homotopy_h_ms}
 \phi_{m,s}: (U_m,\partial U_m) 
 \to 
 (\{E<m^2c+m\epsilon\},\{E<m^2c\}),
 \qquad s\in[0,1],
\end{align}
such that 
\begin{itemize}
\item[(i)] $\phi_{m,0}=\phi_m$,
\item[(ii)] $\phi_{m,1}(U_m)\subset \{E<m^2c\}$.
\end{itemize}
\end{lem}

\begin{proof}
Let $\epsilon_0>0$ be such that
\begin{align}
\label{e:energy_boundary}
\max E\circ\phi|_{\partial(\Sigma_0\times B')} & = c-\epsilon_0.
\end{align}
For each $\epsilon\in(0,\epsilon_0)$, there exist $\delta_\epsilon>0$ small enough so that
\begin{align}
\label{e:delta}
 \rho(3\delta_\epsilon) & <\epsilon,
\end{align}
where $\rho$ is the monotone function introduced in Equation~\eqref{e:control_f}. Since the composition $E\circ\phi$ has a unique global maximum at $(q_*,p_*)$, there exists $\mu_\epsilon>0$ such that
\begin{align}
\label{e:margin}
\max\big\{E\circ\phi(q,p)\ \big|\ 
(q,p)\in B^d\times B',\ \dist(q,q_*)\geq\delta_\epsilon
\big\}
=
c-\mu_\epsilon.
\end{align}
Notice that $\mu_\epsilon\to0$ as $\epsilon\to0$.

For $\epsilon\in(0,\epsilon_0)$, we fix an arbitrary continuous homotopy 
\[\theta_s=(\alpha_s,\beta_s):B^d\to B^d\times B^d,\qquad s\in[0,1],\]
such that $\alpha_0=\beta_0=\mathrm{id}$, $\alpha_s|_{\partial B^d}=\beta_s|_{\partial B^d}=\mathrm{id}$, and for all $q\in B^d$ we have
\begin{gather}
\label{e:homotopy_1}
 \max_{s\in[0,1]} \dist(\alpha_s(q),\beta_s(q))\leq 3\delta_\epsilon,\\
\label{e:homotopy_2}
 \min\big\{\dist(\alpha_1(q),q_*),\dist(\beta_1(q),q_*)\big\}\geq\delta_\epsilon,
\end{gather}
see Figure~\ref{f:homotopy}. By~\eqref{e:margin} and~\eqref{e:homotopy_2}, we have
\begin{align}
\label{e:estimate_below}
\max_{(q,p)\in B^d\times B'}\min\big\{E\circ\phi(\alpha_1(q),p),E\circ\phi(\beta_1(q),p)\big\}\leq c-\mu_\epsilon.
\end{align}

\begin{figure}
\begin{center}
\begin{footnotesize}
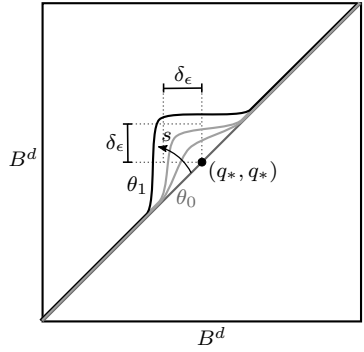 
\end{footnotesize}
\begin{small}
\caption{The homotopy $\theta_s$.}
\label{f:homotopy}
\end{small}
\end{center}
\end{figure}

For $m\in\PP$, we consider the cycle $\phi_m:U_m\to\Sigma_m\subset\Lambda_{km}M$ generating the local homology $\Loc_*(\gamma^m)$, and the map $\Phi_m: \Sigma_0\times \Sigma_0\times (B')^{\times m}\to \Lambda_{mk}M$ defined above. We define the homotopy
\begin{align*}
& \phi_{m,s}:U_m\to\Sigma_m\subset\Lambda_{km}M,
\qquad s\in[0,1],\\
& \phi_{m,s}(q,\pp):=
\Phi_m(\alpha_s(q),\beta_s(q),\pp),
\end{align*}
which clearly satisfies point (i) of the Lemma.

Fix an arbitrary $(q,\pp)\in\partial U_m$, so that $(q,p_i)\in\partial(\Sigma_0\times B')$ for some $i\in\{0,...,m-1\}$. By~\eqref{e:energy_boundary} we have
$E\circ\phi(q,p_i)\leq c-\epsilon_0$,
whereas $E\circ\phi(q,p_j)\leq c$ for all $j\neq i$.
By \eqref{e:energy_Phi_m}, \eqref{e:control_f}, \eqref{e:energy_boundary}, \eqref{e:delta}, and \eqref{e:homotopy_1}, for all $s\in[0,1]$ we have
\begin{align*}
 m^{-1}\,E\circ \phi_{m,s}(q,\pp) = & \sum_{i=0}^{\lfloor m/2\rfloor-1}E\circ\phi(\alpha_s(q),p_i) \\  
 & 
 +\sum_{i=\lfloor m/2\rfloor}^{m-1}E\circ\phi(\beta_s(q),p_i) \\
 & + 
f(\alpha_s(q),p_{\lfloor m/2\rfloor-1},\beta_s(q),p_{\lfloor m/2\rfloor})\\
\leq &\, mc-\epsilon_0 + \rho\big(\dist(\alpha_s(q),\beta_s(q))\big)\\
\leq &\, mc-\epsilon_0 + \rho(3\delta_\epsilon)\\
< &\, mc.
\end{align*}
Analogously, for all $(q,\pp)\in U_m$ we have
$ E\circ \phi_{m,s}(q,\pp) < m^2c+m\epsilon$.
These estimates show that each map $\phi_{m,s}$ in the homotopy has the form~\eqref{e:homotopy_h_ms}.

By~\eqref{e:estimate_below}, for all $(q,\pp)\in U_m$ we have
\begin{align*}
 m^{-1}\,E\circ \phi_{m,1}(q,\pp) = & \sum_{i=0}^{\lfloor m/2\rfloor-1}E\circ\phi(\alpha_1(q),p_i) \\  
 & 
 +\sum_{i=\lfloor m/2\rfloor}^{m-1}E\circ\phi(\beta_1(q),p_i) \\
 & + 
f(\alpha_1(q),p_{\lfloor m/2\rfloor-1},\beta_1(q),p_{\lfloor m/2\rfloor})\\
< &\, mc-(\lfloor m/2\rfloor-1)\mu_\epsilon + \rho(3\delta_\epsilon)\\
\leq &\, mc-(\lfloor m/2\rfloor-1)\mu_\epsilon + \epsilon,
\end{align*}
and therefore
\begin{align*}
 E\circ \phi_{m,1}(q,\pp) < m^2c,
 \qquad
 \forall m > \overline m_\epsilon:=2\left(\frac{\epsilon}{\mu_\epsilon}+1\right),
\end{align*}
which is exactly point (ii) of the lemma.
\end{proof}

\subsection{From local to global}\label{ss:local_to_global}
We now combine the local deformation built in the previous subsection with Benci-Giannoni's penalization method of Section~\ref{s:Benci_Giannoni} in order to detect other closed geodesics, and ultimately prove Theorem~\ref{t:main} in the case where $\avind(\gamma)>0$.

Since $\avind(\gamma)>0$, the iteration inequality \eqref{e:iteration_inequality_1} implies that 
\begin{align*}
d_m +1 > \dim(M),
\qquad
\forall m> \tilde m:=\frac{2\dim(M)}{\avind(\gamma)}.
\end{align*}
We consider the constants $\overline m_\epsilon>0$ of Lemma~\ref{l:homotopy_h_ms}, and we set
\begin{align*}
 \ell:=\sqrt c=\mathrm{length}(\gamma).
\end{align*}

\begin{lem}
\label{l:new_geodesics}
If all the closed geodesics of $(M,g)$ are contained in a compact subset of $M$, for each $\epsilon>0$ small enough and for each integer $m>  \max\{\overline m_\epsilon,\tilde m\}$ there exists a (possibly iterated) closed geodesic $\zeta_m\in\crit(E)$ such that
\begin{align}
\label{e:length_zeta_m}
m\ell < \mathrm{length}(\zeta_m) \leq m\ell\,\big(1+\tfrac{\epsilon}{m\ell^2}\big)^{1/2}\leq m\ell+\tfrac{\epsilon}{2\ell}.
\end{align}
\end{lem}

\begin{proof}
By our assumption, there exists a relatively compact open subset $W\subset M$ such that
\begin{align*}
 \crit(E)\subset\Lambda W.
\end{align*}
Let us consider the sequence of penalized functions  $\{F_\alpha:\Lambda M\to\R\ |\ \alpha\in\N\}$ constructed in Section~\ref{s:Benci_Giannoni}, which have the following properties: there exists an exhaustion by compact subsets $\{K_\alpha\ |\ \alpha\in\N\}$, with  $K_\alpha\subset\mathrm{int}(K_{\alpha+1})$ and $\bigcup_{\alpha\in\N}K_\alpha=M$, such that $F_\alpha|_{\Lambda K_\alpha}=E|_{\Lambda K_\alpha}$, and for all $\alpha\in\N$ large enough $F_\alpha$ satisfies the Morse inequality
\begin{equation}
\label{e:final_Morse_inequality}
\rank \Hom_{d_m+1}(\{F_\alpha<m^2c+m\epsilon\},\{F_\alpha\leq m^2c\})
\leq 
\sum_{S^1\cdot\zeta} \rank\Loc_{d_m+1}(E,S^1\cdot\zeta),
\end{equation}
where the sum in the right-hand side ranges over all the critical circles 
\[S^1\cdot\zeta\subset\crit(E)\cap E^{-1}(m^2c,m^2c+m\epsilon).\]

\begin{rem}
\label{r:avind_zero}
If the closed geodesic $\gamma$ had zero average index $\avind(\gamma)=0$, we would have $d_m+1=d+1=\dim(M)$ for all $m\in\N$, and therefore we would not necessarily  have the Morse inequality~\eqref{e:final_Morse_inequality}. See Section~\ref{s:Benci_Giannoni}.
\hfill\qed
\end{rem}

Consider an integer $m>  \max\{\overline m_\epsilon,\tilde m\}$. We choose $\alpha\in\N$ large enough so that $\Lambda K_\alpha$ contains the image of the homotopy $\phi_{m,s}$, $s\in[0,1]$, provided by Lemma~\ref{l:homotopy_h_ms}. Therefore, in such a lemma we can replace the energy function $E$ with the penalized function $F_\alpha$. By Lemma~\ref{l:homotopy_h_ms}(ii),  the homomorphism
\begin{align*}
  (\phi_{m})_*=(\phi_{m,0})_*=(\phi_{m,1})_*:
  \Hom_*(U_m,\partial U_m) 
 \to 
 \Hom_*(\{F_\alpha<m^2c+m\epsilon\},\{F_\alpha<m^2c\})
\end{align*}
is the zero one, and we have the commutative diagram
\begin{align*}
\xymatrix{
\Hom_{d_m}(U_m,\partial U_m)
\ar[r]_{\cong}^{(\phi_m)_*}
\ar[ddr]_{(\phi_m)_*=0\ }
&
\Loc_{d_m}(E,\gamma)\,
\ar@{^{(}->}[r]^{\incl}_{(*)}
&
\Loc_{d_m}(E,S^1\cdot\gamma)
\ar[ddl]^{\incl}_{(**)}
\\\\
&
\Hom_{d_m}(\{F_\alpha<m^2c+m\epsilon\},\{F_\alpha<m^2c\})
&
} 
\end{align*}
Here, the fact that the homomorphism $(*)$ induced by the inclusion is injective follows from Corollary~\ref{c:incl_local_homologies}. The commutative diagram implies that the homomorphism $(**)$ is not injective. Since the inclusion induces an injective homomorphism 
\begin{align*}
\Loc_*(E,S^1\cdot\gamma^m) \eembup^{\incl} \Hom_*(\{F_\alpha\leq m^2c\},\{F_\alpha<m^2c\}),
\end{align*}
we actually have that the homomorphism 
\begin{align*}
 \Hom_{d_m}(\{F_\alpha\leq m^2c\},\{F_\alpha<m^2c\}) \ttoup^{\incl} \Hom_{d_m}(\{F_\alpha<m^2c+m\epsilon\},\{F_\alpha<m^2c\})
\end{align*}
is not injective. This, together with the long exact sequence of the triple 
\[\big(\{F_\alpha<m^2c+m\epsilon\},\{F_\alpha\leq m^2c\},\{F_\alpha<m^2c\}\big),\] 
implies that
\[\Hom_{d_m+1}(\{F_\alpha<m^2c+m\epsilon\},\{F_\alpha\leq m^2c\})\neq0.\] 
By the Morse inequality~\eqref{e:final_Morse_inequality}, there exists a (possibly iterated) closed geodesic $\zeta_m\in\crit(E)$ such that
\begin{gather}
\label{e:energy_zeta_m}
m^2c < E(\zeta_m) \leq m^2c+m\epsilon.
\end{gather}
The energy bounds~\eqref{e:energy_zeta_m} imply the length bounds~\eqref{e:length_zeta_m}.
\end{proof}

\begin{proof}[Proof of Theorem~\ref{t:main}, case $\avind(\gamma)>0$]
If no compact subset of $M$ contains all the closed geodesics of $(M,g)$, in particular there are infinitely many closed geodesics, and we are done. Therefore, it remains to consider the case where all the closed geodesics of $(M,g)$ are contained in a compact subset of $M$.

Let us fix a sequence of positive quantities $\epsilon_\alpha>0$ such that $\epsilon_\alpha\to0$ as $\alpha\to\infty$. We set 
\[m_\alpha:=\min\big\{ m\in\PP\ \big|\ m>\max\{\overline m_{\epsilon_\alpha},\tilde m,\lceil1/\alpha\rceil\} \big\},\] 
and we consider the closed geodesic $\zeta_{m_\alpha}$ provided by Lemma~\ref{l:new_geodesics}. In order to prove Theorem~\ref{t:main}, we are only left to show that the family of closed geodesics $\{\zeta_{m_\alpha}\ |\ \alpha\in\N\}$ is not made of iterates of finitely many closed geodesics. Actually, we will show that only finitely many iterates of any given closed geodesic of $(M,g)$ can belong to the family $\{\zeta_{m_\alpha}\ |\ \alpha\in\N\}$. 

Let us assume by contradiction that there exists a closed geodesic $\zeta:\R/\Z\to M$, an infinite subset $\K\subset\N$, and a function $\mu:\K\to\N$ such that $\zeta_{m_\alpha}=\zeta^{\mu(\alpha)}$ for all $\alpha\in\K$. The length bounds~\eqref{e:length_zeta_m} can be rewritten as 
\begin{align*}
 m_\alpha\ell < \mu(\alpha)\,\mathrm{length}(\zeta) \leq m_\alpha\ell+\tfrac{\epsilon_\alpha}{2\ell},
\end{align*}
and, since $\mu(\alpha)$ is a positive integer,
\begin{align}
\label{e:final_inequality}
\frac{m_\alpha\,\ell}{\mathrm{length}(\zeta)}
<
\mu(\alpha)
\leq
\left\lfloor 
\frac{m_\alpha\,\ell}{\mathrm{length}(\zeta)}
+
\frac{\epsilon_\alpha}{2\ell\,\mathrm{length}(\zeta)}
\right\rfloor.
\end{align}
However, if we choose $\alpha\in\K$ large enough, we have
\begin{align*}
\left\lfloor 
\frac{m_\alpha\,\ell}{\mathrm{length}(\zeta)}
+
\frac{\epsilon_\alpha}{2\ell\,\mathrm{length}(\zeta)}
\right\rfloor
=
\left\lfloor 
\frac{m_\alpha\,\ell}{\mathrm{length}(\zeta)}
\right\rfloor
\leq
\frac{m_\alpha\,\ell}{\mathrm{length}(\zeta)},
\end{align*}
which, together with~\eqref{e:final_inequality}, gives the contradiction
\[
\frac{m_\alpha\,\ell}{\mathrm{length}(\zeta)}
<
\mu(\alpha)
\leq
\frac{m_\alpha\,\ell}{\mathrm{length}(\zeta)}.
\qedhere
\]
\end{proof}

\section{The case of zero average index}
\label{s:avind_zero}

\subsection{Persistence of the local homology}
Let us assume that we are in the same setting of the previous section, but now the closed geodesic $\gamma\in\crit(E)$ satisfying the assumptions of Theorem~\ref{t:main} has zero average index
\begin{align*}
 \avind(\gamma)=0.
\end{align*}
As we explained in Remark~\ref{r:avind_zero}, the arguments of the previous section does not allow to prove Theorem~\ref{t:main} anymore. Instead, we will follow Bangert and Klingenberg's strategy as in \cite{Bangert:1983ax}.

Since  Bott's function $N:S^1\to\N\cup\{0\}$ associated to $\gamma$ has finite support, up to removing finitely many numbers from the infinite set of primes $\PP$ of Theorem~\ref{t:main} we have $N(z)=0$ for all $z\in\sqrt[m]1$. Bott's formula~\eqref{e:Bott_formula_nullity} thus implies
\begin{align*}
 \nul(\gamma^m)=\nul(\gamma),
 \qquad
 \forall m\in\PP,
\end{align*}
and Lemma~\ref{l:loc_isom_iteration} implies that the iteration map induces a local homology isomorphism
\begin{align*}
\psi^m_*: \Loc_*(E,S^1\cdot\gamma) \toup^{\cong} \Loc_*(E,S^1\cdot\gamma^m).
\end{align*}
By our assumption, the local homology $\Loc_d(E,\gamma)$ is non-trivial in degree 
\[
d=\dim(M)-1=\ind(\gamma)+\nul(\gamma)\geq1,
\] 
which is the maximal degree where it can be non-trivial. By Lemma~\ref{l:exact_sequence}, we have
\begin{align*}
 \Loc_{d+1}(E,S^1\cdot\gamma) 
 \cong
 \Loc_{d+1}(E,\gamma) \oplus  \Loc_{d}(E,\gamma) 
 \cong
 \Loc_{d}(E,\gamma)\neq 0.
\end{align*}

\subsection{Bangert and Klingenberg's homotopies}
We set $c:=E(\gamma)$. Let $[\sigma]$ be a non-zero element in the local homology group $\Loc_{d+1}(E,S^1\cdot\gamma)$, which is a linear combination
\begin{align*}
 \sigma=\lambda_1\sigma_1 +...+ \lambda_k\sigma_k.
\end{align*}
with coefficients $\lambda_i\in\Q$ and singular simplexes $\sigma_i:\Delta^{d+1}\to\{E<c\}\cup S^1\cdot\gamma$. We recall that the local homology of a critical circle that is an isolated local minimizer of a function is isomorphic to the homology of the circle. Since $d+1>1$, we infer that the critical circle $S^1\cdot\gamma$ is not a local minimizer of the energy function $E$. Therefore, we can apply a well known construction of Bangert and Klingenberg \cite[Th.~2]{Bangert:1983ax} that provides homological vanishing under iterations: for each sufficiently large $m\in\N$ there exist homotopies
$\sigma_{i,m,s}:\Delta^{d+1}\to\Lambda M$, $s\in[0,1]$, such that
\begin{itemize}
\item[(i)] $\sigma_{i,m,0}=\psi^m\circ\sigma_i$,
\item[(ii)] $E\circ\sigma_{i,m,1}<E(\gamma^m)=m^2c$,
\item[(iii)] If $\Delta\subset\Delta^{d+1}$ is a $j$-dimensional face (for some $j\in\{0,...,d+1\}$) of the standard simplex such that $E\circ\sigma_{i,m,0}|_{\Delta}<E(\gamma^m)$, then $\sigma_{i,m,s}|_{\Delta}=\sigma_{i,m,0}|_{\Delta}$ for all $s\in[0,1]$.
\end{itemize}
The union of the image of all homotopies $\sigma_{i,m,s}$ is contained in the free-loop space of a compact subset $W=W(m)\subset M$. Namely,
\begin{align*}
W:=
\big\{
\sigma_{i,m,s}(x)(t)
\ \big|\ i=1,...,k,\ s\in[0,1],\ x\in\Delta^{d+1},\ t\in\R/\Z
\big\},
\end{align*}
so that
\begin{align*}
\sigma_{i,m,s}(\Delta^{d+1})\subset \Lambda W,\qquad
\forall i=1,...,k,\ s\in[0,1].
\end{align*}
We fix a large enough constant $b=b(m)\in\R$ such that
\begin{align*}
 b > \max\big\{E\circ\sigma_{i,m,s}(x)\ \big|\ i=1,...,k,\ s\in[0,1],\ x\in\Delta^{d+1}  \big\}
 \geq E(\gamma^m)=m^2c.
\end{align*}
We consider the relative chains
\begin{align*}
\sigma_{m,s}:=\lambda_1\sigma_{1,m,s} +...+ \lambda_k\sigma_{k,m,s},
\end{align*}
which, by point (iii), represent the same homology class 
\begin{equation*}
\begin{split}
 [\sigma_{m,s}]=[\sigma_{m,0}]=\psi^m_*[\sigma]\in\Hom_{d+1}(\{E|_{\Lambda W}<b\},\{E|_{\Lambda W}<m^2c\}),
\\
\forall s\in[0,1]
\end{split}
\end{equation*}
Point~(ii) actually implies that $[\sigma_{m,1}]=0$, and thus
\begin{align}
\label{e:hom_vanishing}
 [\sigma_{m,s}]=0\in\Hom_{d+1}(\{E|_{\Lambda W}<b\},\{E|_{\Lambda W}<m^2c\}),
 \qquad\forall s\in[0,1].
\end{align}

\subsection{From local to global}
We now proceed as in Subsection~\ref{ss:local_to_global}: up to enlarging the compact subset $W\subset M$, we can assume that all the closed geodesics of $(M,g)$ are contained in the interior of $W$ (in the opposite case Theorem~\ref{t:main} holds without further work). We fix $m\in\PP$ large enough so that the homotopies of the previous subsection are defined. We consider the sequence  $\{F_\alpha:\Lambda M\to\R\ |\ \alpha\in\N\}$ of penalized functions constructed in Section~\ref{s:Benci_Giannoni}, so that for all $\alpha\in\N$ large enough $F_\alpha|_{\Lambda W}=E|_{\Lambda W}$ and we have the Morse inequality
\begin{equation}
\label{e:final_Morse_inequality_avind_zero}
\rank \Hom_{d+2}(\{F_\alpha<b\},\{F_\alpha\leq m^2c\})
\leq 
\sum_{S^1\cdot\zeta} \rank\Loc_{d+2}(E,S^1\cdot\zeta),
\end{equation}
where the sum in the right-hand side ranges over all the critical circles 
\[S^1\cdot\zeta\subset\crit(E)\cap E^{-1}(m^2c,b).\]
Notice that this is possible, since $d+2=\dim(M)+1>\dim(M)$, see Section~\ref{s:Benci_Giannoni}.
Equation~\eqref{e:hom_vanishing} readily implies that
\begin{align*}
[\sigma_{m,s}]=0\in\Hom_{d+1}(\{F_\alpha<b\},\{F_\alpha<m^2c\}),
\qquad\forall s\in[0,1], 
\end{align*}
Consider the inclusion induced homomorphism 
\begin{align*}
\Loc_{d+1}(E,S^1\cdot\gamma^m)=\Loc_{d+1}(F_\alpha,S^1\cdot\gamma^m) \ttoup^{\incl} \Hom_{d+1}(\{F_\alpha<b\},\{F_\alpha<m^2c\}).
\end{align*}
This homomorphism is not injective: indeed, by point~(i), it maps the non-zero element $\psi^m_*[\sigma]$ to $[\sigma_{m,0}]=0$. Since the inclusion induces an injective homomorphism
\begin{align*}
\Loc_*(E,S^1\cdot\gamma^m)=\Loc_*(F_\alpha,S^1\cdot\gamma^m) \eembup^{\incl} \Hom_*(\{F_\alpha\leq m^2c\},\{F_\alpha<m^2c\}),
\end{align*}
the long exact sequence of the triple
\begin{align*}
(\{F_\alpha<b\},\{F_\alpha\leq m^2c\},\{F_\alpha<m^2c\})
\end{align*}
implies that $\Hom_{d+2}(\{F_\alpha<b\},\{F_\alpha\leq m^2c\})$ is non-trivial. The Morse inequality~\eqref{e:final_Morse_inequality_avind_zero} implies that there exists a closed geodesic $\zeta_m\in\crit(E)$ with 
\begin{align}\label{e:avindzero_energy_bounds}
m^2c<E(\zeta_m)<b 
\end{align}
and with non-trivial local homology $\Loc_{d+2}(E,S^1\cdot\zeta_m)\neq0$. In particular, 
\begin{align}\label{e:avindzero_index_bounds}
\ind(\zeta_m)\leq d+2=\dim(M)+1\leq \ind(\zeta_m)+\nul(\zeta_m)+1.
\end{align}

\begin{proof}[Proof of Theorem~\ref{t:main}, case $\avind(\gamma)=0$]
If the closed geodesics of $(M,g)$ are not contained in a compact subset of $M$, we are done. In the other case, the above argument gives us a sequence $\{\zeta_m\ |\ m\in\N\}\subset\crit(E)$ satisfying the energy bounds~\eqref{e:avindzero_energy_bounds} and the index bounds~\eqref{e:avindzero_index_bounds}. The iteration inequality~\eqref{e:iteration_inequality_2} implies that \[\avind(\zeta_m)>0.\]
We claim that the sequence $\{\zeta_m\ |\ m\in\N\}$ is not comprised of  iterates of only finitely many closed geodesics, and that indeed only finitely many iterations of a given closed geodesic $\zeta\in\crit(E)$ may belong to the sequence $\{\zeta_m\ |\ m\in\N\}$. 

If this latter claim does not hold, there exists a function $\mu:\PP\to\N$ such that $\zeta_m=\zeta^{\mu(m)}$. Since 
\[ \mu(m)E(\zeta)= E(\zeta^{\mu(m)})= E(\zeta_m)> m^2 c, \]
we have that $\mu(m)\to\infty$ as $m\to\infty$. This, together with the iteration inequality \eqref{e:iteration_inequality_1} and
\begin{align*}
\avind(\zeta)=\frac{\avind(\zeta^{\mu(m)})}{\mu(m)}=\frac{\avind(\zeta_m)}{\mu(m)}>0,
\end{align*}
implies that 
\[\ind(\zeta_m)=\ind(\zeta^{\mu(m)})\geq \mu(m)\,\avind(\zeta)-(\dim(M)-1)\tttoup^{m\to\infty}\infty,\] 
contradicting the first inequality in~\eqref{e:avindzero_index_bounds}.
\end{proof}
\appendix

\section{Local homology}

\subsection{Homotopic invariance of the local homology}

The time-1 map of a deformation retraction induces an isomorphism in singular homology. For the local homology of a subset in a topological space, one only needs a weaker assumption on the homotopy.

\begin{prop}\label{p:homotopic_invariance_local_homology}
Let $X$ be a topological space, and $Z\subset Y\subset X$ two subspaces. Assume that, for some open neighborhood $W\subset X$ of $Z$, there exists a homotopy 
\[h_s:(W,W\setminus Z)\to (X,X\setminus Z),\qquad s\in[0,1],\] 
such that $h_0=\mathrm{id}$, $h_1|_{Y\cap W}=\mathrm{id}$, and $h_1(W)\subset Y$. Then, the inclusion induces the local homology isomorphism
\begin{align*}
\Hom_*(Y,Y\setminus Z) \ttoup^{\incl}_{\cong} \Hom_*(X,X\setminus Z).
\end{align*}
\end{prop}

\begin{proof}
The proof is a direct consequence of the following commutative diagram
\begin{align*}
\xymatrix{
\Hom_*(Y,Y\setminus Z)
\ar[rr]^{\incl} 
& &
\Hom_*(X,X\setminus Z)\\\\
\Hom_*(Y\cap W,Y\cap W\setminus Z)
\ar[rr]^{\incl} 
\ar[uu]^{\incl}_{\cong}
& &
\Hom_*(W,W\setminus Z)
\ar[uu]_{\incl}^{\cong}
\ar[lluu]_{(h_1)_*} 
} 
\end{align*}
where the commutativity of the lower-left triangle follows from the fact that $h_1$ fixes $Y\cap W$, the commutativity of the upper-right triangle follows from the fact that $h_1$ is the time-1 map of the homotopy $h_s:(W,W\setminus Z)\to(X,X\setminus Z)$ with $h_0=\mathrm{id}$, and the vertical homomorphisms are isomorphisms according to the excision property.
\end{proof}

\subsection{Local homology homomorphisms}
The following statement is certainly well known to the experts.

\begin{prop}\label{p:local_homology_maps}
Let $U\subseteq U'\subseteq \R^n\times\R^m$ be open neighborhoods of the origin $\bm0=(0,0)$, $F:U'\to\R$ a smooth function, and $\phi:U\to U'$ a diffeomorphism onto its image, verifying the following assumptions:
\begin{itemize}
\item[(i)] the origin $\bm0$ is an isolated critical point of $F$;

\item[(ii)] $F(x,v)=F(x,0)-\tfrac12 |v|^2$ for all $(x,v)\in U'$;

\item[(iii)] $F\circ\phi\leq F|_U$ in a neighborhood of $\bm0$;

\item[(iv)] $\phi(x,0)=(x,0)$ for all $(x,0)\in U$.
\end{itemize}
Then $\phi$ induces the local homology isomorphism
\[\phi_*=(\mathrm{sign}(\det\diff\phi(\bm0)))\mathrm{id}:\Loc_*(F,\bm0)\to\Loc_*(F,\bm0).\]
\end{prop}

\begin{proof}
Let us write $f(x):=F(x,0)$ and $\phi(x,v)=(\alpha(x,v),\beta(x,v))$ for suitable smooth maps $\alpha$ and $\beta$. Notice that
\begin{align*}
\diff\phi(x,0)
=
\left(
  \begin{array}{cc}
    \partial_x\alpha(x,0) & \partial_v\alpha(x,0) \\ 
    \partial_x\beta(x,0) & \partial_v\beta(x,0) 
  \end{array}
\right)
=
\left(
  \begin{array}{cc}
    \mathrm{id} & \partial_v\alpha(x,0) \\ 
    0 & \partial_v\beta(x,0) \\ 
  \end{array}
\right),
\end{align*}
and therefore the map $\beta_x(v):=\beta(x,v)$ is a local diffeomorphism at the origin, for all $x\in\R^m$ such that $(x,0)\in U$. If $V\subset\R^n$ is a sufficiently small neighborhood of the origin, the map $\beta_0:V\setminus\{0\}\to\R^n\setminus\{0\}$ is homotopic to $(\mathrm{sign}(\det\diff\phi(\bm0)))\mathrm{id}$ through maps of the form $V\setminus\{0\}\to\R^n\setminus\{0\}$.

Up to shrinking the neighborhood $U$, there exist constants $b_2>b_1>0$, $c_1>0$, and $c_2>0$ such that $b_1\|v\|\leq\|\beta(x,v)\|\leq b_2\|v\|$, $\|\diff f(x)\|\leq c_1$, and $\|\diff\alpha(x,v)\|\leq c_2$ for all $(x,v)\in U$. We fix a constant $\lambda>0$ such that 
\begin{align*}
e^{2\lambda}
>
\frac{b_2^2+2c_1c_2}{b_1^2}.
\end{align*}
We shrink $U$ around the origin, so that the homotopy $\phi_s:U\to\R^n\times\R^m$, given by 
\begin{align*}
\phi_s(x,v)
=
\left(\alpha(x,v),\frac{e^{s\lambda}}{1-s+s\|v\|^{1/2}}\beta(x,v)\right),
\qquad
s\in[0,1],
\end{align*}
is well defined and continuous. We also require $U$ to be small enough so that
\begin{align*}
 \|v\|<1,\qquad\forall (x,v)\in U.
\end{align*}
Notice that $\phi_0=\phi$, and for all $s\in[0,1]$ we have
\begin{align}
\label{e:condition_on_homotopy_1}
 &\phi_s(x,0)=(x,0),\qquad\forall(x,0)\in U,\\
\label{e:condition_on_homotopy_2}
 &F\circ\phi_s\leq F.
\end{align}
We extend the homotopy $\phi_s$ for $s\in[1,2]$ by setting
\begin{align*}
\phi_s(x,v)
=
\left(\alpha(x,(2-s)v),e^{\lambda}\|v\|^{-1/2}\beta((2-s)x,v)\right),
\qquad
s\in[1,2].
\end{align*}
Condition~\eqref{e:condition_on_homotopy_1} still holds for all $s\in[1,2]$. Condition~\eqref{e:condition_on_homotopy_2} also holds for all $s\in[1,2]$, for
\begin{align*}
F\circ\phi(x,v)
=
f(\alpha(x,v)) - \frac{\|\beta(x,v)\|^2}{2}
\leq
F(x,v),
\end{align*}
and therefore
\begin{align*}
F\circ\phi_s(x,v)
&=
f(\alpha(x,(2-s)v))
-
\frac{e^{2\lambda}\|\beta((2-s)x,v)\|^2}{\|v\|}\\
&\leq
f(\alpha(x,v)) + c_1c_2 \|(s-1)v\| - \frac{e^{2\lambda}b_1^2}{2}\|v\|\\
&\leq
F(x,v)+\frac{\|\beta(x,v)\|^2}{2} + c_1c_2 \|(s-1)v\| - \frac{e^{2\lambda}b_1^2}{2}\|v\|\\
&\leq
F(x,v)+\frac{b_2^2}{2}\|v\|^2 + c_1c_2 \|(s-1)v\| - \frac{e^{2\lambda}b_1^2}{2}\|v\|\\
&\leq
F(x,v)+\underbrace{\left(\frac{b_2^2}{2} + c_1c_2 - \frac{e^{2\lambda}b_1^2}{2}\right)}_{<0}\|v\|\\
& \leq F(x,v).
\end{align*}
The final map of this homotopy has the form $\phi_2(x,v)=(x,\kappa(v))$, where 
\begin{align*}
\kappa(v)=e^\lambda \|v\|^{-1/2} \beta(0,v).
\end{align*}
In particular
\begin{align*}
k_1\|v\|^{1/2} \leq \|\kappa(v)\|\leq k_2\|v\|^{1/2}
\end{align*}
for some constants $k_2>k_1>0$. We shrink $U$ to be of the form $U=B'\times B''$, where $B'\subset\R^n$ is a neighborhood of the origin, and $B''\subset\R^m$ is the open ball of radius $2k_0$, for some $k_0\in(0,k_1^2)$. We continuously extend the homotopy $\phi_s$ for $s\in[2,3]$ by setting
\begin{align*}
\phi_s(x,v)
=
\left( x , (3-s+(s-2) k_0^{-1}\|v\|)\, \kappa\left( \frac{v}{3-s+(s-2) k_0^{-1}\|v\|}\right) \right)\!,
\quad
s\in[2,3].
\end{align*}
For all $(x,v)\in U$ with $\|v\|\leq k_0$ and $s\in[2,3]$, we have
\begin{align*}
F\circ\phi_s(x,v)
&=
f(x)
-\frac{1}{2} \Big|3-s+(s-2) k_0^{-1}\|v\|\Big|^2 \left\|
\kappa\left( 
\frac{v}{3-s+(s-2) k_0^{-1}\|v\|}\right) 
\right\|^2\\
&\leq
f(x)
-\frac{1}{2} k_1^2 \|v\| \Big|3-s+(s-2) k_0^{-1}\|v\|\Big|\\
&\leq
f(x)
-\frac{k_1^2}{2k_0}  \|v\|^2 \\
&\leq
f(x)
-\frac{1}{2}  \|v\|^2 \\
&=
F(x,v).
\end{align*}
Thus, up to further shrinking the neighborhood $U$, conditions~\eqref{e:condition_on_homotopy_1} and ~\eqref{e:condition_on_homotopy_2} hold for all $s\in[2,3]$. The final map of the homotopy has the form $\phi_3(x,v)=(x,\omega(v))$, where 
\begin{align*}
 \omega(v) = k_0^{-1} \|v\|\, \kappa( k_0 \|v\|^{-1}v ).
\end{align*}
Notice that $\omega$ is positively homogeneous of degree 1, i.e. 
\[\omega(\lambda v)=\lambda\omega(v),
\qquad\forall \lambda >0,\]
and satisfies $\|\omega(v)\|\geq\|v\|$. We continuously extend the homotopy $\phi_s$ for $s\in[3,4]$ by setting
\begin{align*}
\phi_s(x,v)
=
\left( x , \frac{\omega(v)}{4-s+(s-3)\|\omega(v)\|\,\|v\|^{-1}}\right)\!,
\quad
s\in[3,4].
\end{align*}
For all $s\in[3,4]$, we have
\begin{align*}
F\circ\phi_s(x,v)
&=
f(x)
-\frac{1}{2} \frac{\|\omega(v)\|^2}{(4-s+(s-3)\|\omega(v)\|\,\|v\|^{-1})^2}\\
&\leq
f(x)
-\frac{1}{2} \frac{\|\omega(v)\|^2}{(\|\omega(v)\|\,\|v\|^{-1})^2}\\
&= F(x,v).
\end{align*}
and thus conditions~\eqref{e:condition_on_homotopy_1} and ~\eqref{e:condition_on_homotopy_2} hold for all $s\in[3,4]$. The final map has the form $\phi_4(x,v)=(x,\|v\|\,\psi(v/\|v\|))$, where 
\[
\psi:S^{n-1}\to S^{n-1}, 
\qquad
\psi(v)=\frac{\omega(v)}{\|\omega(v)\|}.
\]
Consider again the map $\beta_0:V\to\R^n$ of the beginning of the proof. Notice that the map $v\mapsto\|v\|\,\psi(v/\|v\|)$ is homotopic to the map $\beta_0$ of the beginning of the proof via maps of the form $V\setminus\{0\}\to\R^n\setminus\{0\}$. Thus, there exists a homotopy $\psi_s:S^{n-1}\to S^{n-1}$, for $s\in[4,5]$, such that $\psi_4=\psi$ and 
\[\psi_5(v_1,...,v_n)=\big(\mathrm{sign}(\det\diff\beta_0(0))v_1,v_2,...,v_n\big)=\big(\mathrm{sign}(\det\diff\phi(\bm0))v_1,v_2,...,v_n\big).\] 
We employ it to continuously extend one last time $\phi_s$ for $s\in[4,5]$ as
\begin{align*}
\phi_s(x,v)=(x,\|v\|\psi_s(v/\|v\|)),
\qquad s\in[4,5].
\end{align*}
Once again, this homotopy satisfies~\eqref{e:condition_on_homotopy_1} and~\eqref{e:condition_on_homotopy_2} for all $s\in[4,5]$ (actually we even have $F\circ\phi_s=F$ for all $s\in[4,5]$), and the final map is precisely
\begin{align}
\label{e:form_phi_5}
\phi_5(x,v)=(x,\nu(v)),
\end{align}
where
\begin{align*}
\nu(v_1,v_2,...,v_n)
=
\big(\mathrm{sign}(\det\diff\phi(\bm0))v_1,v_2,...,v_n\big).
\end{align*}
Since conditions~\eqref{e:condition_on_homotopy_1} and~\eqref{e:condition_on_homotopy_2} hold for all $s\in[0,5]$, the homotopy $\phi_s$ restricts as a homotopy of pairs of the form
\begin{align*}
 \phi_s:\big(\{F<F(\bm0)\}\cup\{\bm0\},\{F<F(\bm0)\}\big)\to\big(\{F<F(\bm0)\}\cup\{\bm0\},\{F<F(\bm0)\}\big),
 \\
 \forall s\in[0,5],
\end{align*}
and we have
\[
(\phi_5)_*
=(\phi_0)_*
=\phi_*
:\Loc_*(F,\bm0)\to\Loc_*(F,\bm0).
\]

Now, let $\theta_s$ be the flow of the anti-gradient $-\nabla f$ with respect to some Riemannian metric. By a well known construction in Morse theory due to Gromoll and Meyer  \cite{Gromoll:1969jy}, there exists an arbitrarily small neighborhood $W\subset \R^n$ of the origin with the following property: for any $x\in W$ and $s>0$ such that $\theta_s(x)\not\in W$, there exists $s'\in[0,s)$ such that $\theta_{s'}(x)\in W_-:= W\cap\{f<f(0)\}$. We shrink $U$ to be of the form $U=W\times B$, where $W\subset\R^n$ is one such neighborhood of the origin, and we set $W_*:=W_-\cup\{0\}$. The inclusion
\begin{align*}
\big( W_*\times B,(W_-\times B)\cup (W_*\times B\setminus\{0\})  \big)
\hookrightarrow 
(\{F|_{U}<F(\bm0)\}\cup\{\bm0\},\{F|_{U}<F(\bm0)\})
\end{align*}
induces an isomorphism in homology (indeed, it is even a homotopy equivalence, whose inverse can be easily built by means of the anti-gradient flow $\theta_s$). Since the map $\phi_5$ has the form~\eqref{e:form_phi_5}, by composing the K\"unneth isomorphism
\begin{align*}
K:
\Hom_*(W_*,W_-)\otimes \Hom_*(B,B\setminus\{0\})
\ttoup^{\cong}
\Hom_*\big( W_*\times B,(W_-\times B)\cup (W_*\times B\setminus\{0\})  \big)
\end{align*}
with the latter inclusion-induced isomorphism, we obtain a commutative diagram
\begin{align*}
\xymatrix{
\Hom_*(W_*,W_-)\otimes \Hom_*(B,B\setminus\{0\})
\ar[rr]^{\qquad\qquad K}_{\qquad\qquad\cong}\ar[dd]^{\cong}_{\mathrm{Id_*}\otimes\nu_*}
& &
\Loc_*(F,\bm0)
\ar[dd]^{(\phi_5)_*=\phi_*}_{\cong} 
\\\\
\Hom_*(W_*,W_-)\otimes \Hom_*(B,B\setminus\{0\})
\ar[rr]^{\qquad\qquad K}_{\qquad\qquad\cong}_{\qquad\qquad\cong}
& &
\Loc_*(F,\bm0)
} 
\end{align*}
This completes the proof, since $\nu_*=\mathrm{sign}(\det\diff\phi(\bm0))\mathrm{id}$.
\end{proof}

\bibliography{_biblio}
\bibliographystyle{amsalpha}

\end{document}

%% file: section.pdf_tex
\begingroup%
  \makeatletter%
  \providecommand\color[2][]{%
    \errmessage{(Inkscape) Color is used for the text in Inkscape, but the package 'color.sty' is not loaded}%
    \renewcommand\color[2][]{}%
  }%
  \providecommand\transparent[1]{%
    \errmessage{(Inkscape) Transparency is used (non-zero) for the text in Inkscape, but the package 'transparent.sty' is not loaded}%
    \renewcommand\transparent[1]{}%
  }%
  \providecommand\rotatebox[2]{#2}%
  \ifx\svgwidth\undefined%
    \setlength{\unitlength}{179.41128884bp}%
    \ifx\svgscale\undefined%
      \relax%
    \else%
      \setlength{\unitlength}{\unitlength * \real{\svgscale}}%
    \fi%
  \else%
    \setlength{\unitlength}{\svgwidth}%
  \fi%
  \global\let\svgwidth\undefined%
  \global\let\svgscale\undefined%
  \makeatother%
  \begin{picture}(1,0.79865267)%
    \put(0,0){\includegraphics[width=\unitlength,page=1]{section.pdf}}%
    \put(0.05315722,0.7309354){\color[rgb]{0,0,0}\makebox(0,0)[lb]{\smash{$Z$}}}%
    \put(0.03288788,0.3964912){\color[rgb]{0,0,0}\makebox(0,0)[lb]{\smash{$S^1\cdot\gamma$}}}%
    \put(0.42138357,0.10033351){\color[rgb]{0,0,0}\makebox(0,0)[lb]{\smash{$\Sigma$}}}%
    \put(0,0){\includegraphics[width=\unitlength,page=2]{section.pdf}}%
    \put(0.47993948,0.32892675){\color[rgb]{0,0,0}\makebox(0,0)[lb]{\smash{$\gamma$}}}%
    \put(0.46642659,0.56089814){\color[rgb]{0,0,0}\makebox(0,0)[lb]{\smash{$\zeta$}}}%
    \put(0.49494066,0.67904544){\color[rgb]{0,0,0}\makebox(0,0)[lb]{\smash{$\dot{Z}(0^+)$}}}%
    \put(0.55538657,0.51603624){\color[rgb]{0,0,0}\makebox(0,0)[lb]{\smash{$\dot{Z}(0^-)$}}}%
    \put(0,0){\includegraphics[width=\unitlength,page=3]{section.pdf}}%
  \end{picture}%
\endgroup%

%% file: circle.pdf_tex
\begingroup%
  \makeatletter%
  \providecommand\color[2][]{%
    \errmessage{(Inkscape) Color is used for the text in Inkscape, but the package 'color.sty' is not loaded}%
    \renewcommand\color[2][]{}%
  }%
  \providecommand\transparent[1]{%
    \errmessage{(Inkscape) Transparency is used (non-zero) for the text in Inkscape, but the package 'transparent.sty' is not loaded}%
    \renewcommand\transparent[1]{}%
  }%
  \providecommand\rotatebox[2]{#2}%
  \ifx\svgwidth\undefined%
    \setlength{\unitlength}{309.3359375bp}%
    \ifx\svgscale\undefined%
      \relax%
    \else%
      \setlength{\unitlength}{\unitlength * \real{\svgscale}}%
    \fi%
  \else%
    \setlength{\unitlength}{\svgwidth}%
  \fi%
  \global\let\svgwidth\undefined%
  \global\let\svgscale\undefined%
  \makeatother%
  \begin{picture}(1,0.38053179)%
    \put(0,0){\includegraphics[width=\unitlength,page=1]{circle.pdf}}%
    \put(0.47195345,0.36031459){\color[rgb]{0,0,0}\makebox(0,0)[lb]{\smash{$\gr{I_1\cdot\gamma}$}}}%
    \put(0.47275753,0.00544261){\color[rgb]{0,0,0}\makebox(0,0)[lb]{\smash{$I_2\cdot\gamma$}}}%
    \put(0.24878208,0.18415157){\color[rgb]{0,0,0}\makebox(0,0)[lb]{\smash{$\mu^{1/2}\cdot\gamma$}}}%
    \put(0.6714863,0.18337416){\color[rgb]{0,0,0}\makebox(0,0)[lb]{\smash{$\gamma=\mu\cdot\gamma$}}}%
    \put(0,0){\includegraphics[width=\unitlength,page=2]{circle.pdf}}%
    \put(-0.00092183,0.18337416){\color[rgb]{0,0,0}\makebox(0,0)[lb]{\smash{$\qquad$}}}%
  \end{picture}%
\endgroup%

%% file: homotopy.pdf_tex
\begingroup%
  \makeatletter%
  \providecommand\color[2][]{%
    \errmessage{(Inkscape) Color is used for the text in Inkscape, but the package 'color.sty' is not loaded}%
    \renewcommand\color[2][]{}%
  }%
  \providecommand\transparent[1]{%
    \errmessage{(Inkscape) Transparency is used (non-zero) for the text in Inkscape, but the package 'transparent.sty' is not loaded}%
    \renewcommand\transparent[1]{}%
  }%
  \providecommand\rotatebox[2]{#2}%
  \ifx\svgwidth\undefined%
    \setlength{\unitlength}{132.81279907bp}%
    \ifx\svgscale\undefined%
      \relax%
    \else%
      \setlength{\unitlength}{\unitlength * \real{\svgscale}}%
    \fi%
  \else%
    \setlength{\unitlength}{\svgwidth}%
  \fi%
  \global\let\svgwidth\undefined%
  \global\let\svgscale\undefined%
  \makeatother%
  \begin{picture}(1,0.98122962)%
    \put(0,0){\includegraphics[width=\unitlength,page=1]{homotopy.pdf}}%
    \put(0.56243463,0.48380452){\color[rgb]{0,0,0}\makebox(0,0)[lb]{\smash{$(q_*,q_*)$}}}%
    \put(0.53102633,0.00458822){\color[rgb]{0,0,0}\makebox(0,0)[lb]{\smash{$B^d$}}}%
    \put(-0.00317646,0.51306153){\color[rgb]{0,0,0}\makebox(0,0)[lb]{\smash{$B^d$}}}%
    \put(0,0){\includegraphics[width=\unitlength,page=2]{homotopy.pdf}}%
    \put(0.43095036,0.58611651){\color[rgb]{0,0,0}\makebox(0,0)[lb]{\smash{$s$}}}%
    \put(0,0){\includegraphics[width=\unitlength,page=3]{homotopy.pdf}}%
    \put(0.46661563,0.7553381){\color[rgb]{0,0,0}\makebox(0,0)[lb]{\smash{$\delta_\epsilon$}}}%
    \put(0.27291745,0.56361601){\color[rgb]{0,0,0}\makebox(0,0)[lb]{\smash{$\delta_\epsilon$}}}%
    \put(0.47208189,0.40635929){\color[rgb]{0,0,0}\makebox(0,0)[lb]{\smash{$\gr{\theta_0}$}}}%
    \put(0.33225031,0.44293062){\color[rgb]{0,0,0}\makebox(0,0)[lb]{\smash{$\theta_1$}}}%
    \put(0,0){\includegraphics[width=\unitlength,page=4]{homotopy.pdf}}%
  \end{picture}%
\endgroup%